\documentclass[a4paper,12pt]{article}

\newcommand{\Title}{From Random Lines to Metric Spaces}
\usepackage{ucs}
\usepackage[utf8x]{inputenc}
\usepackage{amsmath}
\usepackage{amsfonts}
\usepackage{amssymb}
\usepackage{amsthm}

\usepackage{a4wide}

\usepackage{verbatim}

\usepackage[T1]{fontenc}
\usepackage{concmath,eulervm}
\usepackage{euscript}
\usepackage{textcomp}

\usepackage[small]{caption}
\usepackage[british]{babel}
\usepackage[pdftex]{graphicx}
\usepackage{calc}
\usepackage{xspace}
\usepackage{framed}

\usepackage[pdftex,colorlinks,bookmarksopen,unicode]{hyperref}
\PrerenderUnicode{Π}
\hypersetup{%
   pdfauthor=Wilfrid S. Kendall,%
}

\usepackage{xcolor}
\input dvipsnam.def
\definecolor{linkcolor}{named}{Maroon}
\definecolor{citecolor}{named}{OliveGreen}
\definecolor{urlcolor}{named}{RoyalPurple}
\definecolor{okcolor}{named}{OliveGreen}
\definecolor{alertcolor}{named}{BrickRed}
 \hypersetup{citecolor=citecolor}
 \hypersetup{linkcolor=linkcolor}
 \hypersetup{urlcolor=urlcolor}

\usepackage[longnamesfirst]{natbib}
\setcitestyle{round}

\newcommand{\ball}{\operatorname{ball}}
\renewcommand{\d}{\text{\rm{d}}}
\newcommand{\dist}{\operatorname{dist}}

\newcommand{\eps}{\varepsilon}

\newcommand{\half}{\frac12}
\newcommand{\hittingset}[1]{\left[#1\right]}
\newcommand{\hits}{\Uparrow}
\newcommand{\origin}{\text{\textbf{o}}}

\newcommand{\Expect}[1]{\operatorname{\mathbb{E}}\left[#1\right]}
\newcommand{\Image}{\operatorname{Im}}

\newcommand{\Leb}{\operatorname{Leb}}
\newcommand{\Lipschitz}{\operatorname{Lip}}
\newcommand{\Lines}{\mathcal{L}}
\newcommand{\Paths}{\mathcal{A}}
\newcommand{\Prob}[1]{\operatorname{\mathbb{P}}\left[#1\right]}
\newcommand{\Reals}{\mathbb{R}}
\newcommand{\Silhouette}{\mathcal{S}}
\newcommand{\Sobolev}{W^{1,2}}

\newtheorem{thm}{Theorem}[section]
\newtheorem{lemma}[thm]{Lemma}
\newtheorem{lem}[thm]{Lemma}

\newtheorem{cor}[thm]{Corollary}
\newtheorem{defn}[thm]{Definition}

\newtheorem{rem}[thm]{Remark}
\numberwithin{equation}{section}

\newenvironment{Figure}%
{\begin{figure}[htbp]\begin{framed}}%
{\end{framed}\end{figure}}

\title{\Title}
\author{Wilfrid S. Kendall\footnote{Work supported in part by EPSRC Research Grant  EP/K013939.}}
\date{\today}

\begin{document}


 \maketitle

\begin{abstract}
Consider an improper Poisson line process, marked by positive speeds so as to satisfy a scale-invariance property (actually, scale-equivariance).
The line process can be characterized by its intensity measure, which belongs to a one-parameter family if scale and Euclidean invariance are required.
This paper investigates a proposal by Aldous, namely that the line process could be used to produce a scale-invariant
random spatial network (SIRSN) by means of connecting up points using paths which follow segments from the line process
at the stipulated speeds. It is shown that this does indeed produce a scale-invariant network, under suitable conditions on the parameter; indeed this then produces
a parameter-dependent random geodesic metric for \(d\)-dimensional space (\(d\geq2\)), where geodesics are given by minimum-time paths. 
Moreover in the planar case it is shown that the resulting geodesic metric space has an almost-everywhere-unique-geodesic property,
that geodesics are locally of finite mean length, and that if an independent Poisson point process is connected up by such geodesics then
the resulting network places finite length in each compact region. It is an open question whether the result is a SIRSN (in Aldous' sense; so placing finite \emph{mean} length in each compact region), 
but it may be called a pre-SIRSN.
\end{abstract}

 \begin{quotation}
  \textsc{Lipschitz path};
  \textsc{marked line process};
  \textsc{orientation field};
  \textsc{perpetuity};
  \textsc{Poisson line process};
  \textsc{pre-SIRSN};
  \textsc{random upper-semicontinuous function};
  \textsc{\(\Pi\)-geodesic};
  \textsc{\(\Pi\)-path};
  \textsc{scale invariance};
  \textsc{SIRSN};
  \textsc{Sobolev space};
  \textsc{spatial network};
  \textsc{stochastic geometry};
  \textsc{weak SIRSN}
 \end{quotation}

 \begin{quotation}\noindent
 \end{quotation}

 \section{Introduction}\label{sec:prelims}
This paper is dedicated to the memory of Don Burkholder, a great probabilist and a kind man.

Recent work in random spatial networks \citep{AldousGanesan-2013,Aldous-2012} has focussed on specification and analysis of an intriguing class of random networks known as
\emph{scale-invariant random spatial networks} (SIRSN). 
Motivated by the success of Google Maps and Bing Maps, \cite{Aldous-2012} shows how a natural collection of desirable properties (statistical invariance under translation, rotation and scale-change,
and some integrability conditions) define a class of models with a useful structure theory.
\begin{defn}[Definition of a SIRSN, \citealp{Aldous-2012}]\label{def:SIRSN}
Consider a \(d\)-dimensional random mechanism, which provides random routes connecting any two points \(x_1,x_2\in\Reals^d\).
We say that this is a \emph{SIRSN} if the following properties hold:
\begin{enumerate}
 \item\label{def:SIRSN-item-route} 
 Between any specified two points \(x_1,x_2\in\Reals^d\), almost surely the random mechanism provides just one connecting random route \(\mathcal{R}(x_1,x_2)=\mathcal{R}(x_2,x_1)\), which is a finite-length path connecting \(x_1\) to \(x_2\).
 \item\label{def:SIRSN-item-invariance} 
 For a finite set of points \(x_1,\ldots,x_k\in\Reals^d\), consider the random network \(\mathcal{N}(x_1,\ldots,x_k)\) formed by the random routes provided by the structure to connect all  \(x_i\) and \(x_j\).
 Then \(\mathcal{N}(x_1,\ldots,x_k)\) is statistically invariant (strictly speaking, equivariant) under translation, rotation, and re-scaling: 
 if \(\mathfrak{S}\) is a Euclidean similarity of \(\Reals^d\) then the networks \(\mathfrak{S}\mathcal{N}(x_1,\ldots,x_k)\) and \(\mathcal{N}(\mathfrak{S}x_1,\ldots,\mathfrak{S}x_k)\)
 have the same distribution.
 \item\label{def:SIRSN-item-finite-length} 
 Let \(D_1\) be the length of the route between two points separated by unit Euclidean distance. Then \(\Expect{D_1}<\infty\).
 \item\label{def:SIRSN-item-locally-finite} 
 Suppose that \(\Xi_\lambda\) is a Poisson point process in \(\Reals^d\), of intensity \(\lambda>0\) and independent of the random mechanism in question.
 Then \(\mathcal{N}(\Xi_\lambda)\), the union of all the networks \(\mathcal{N}(x_1,\ldots,x_k)\) for \(x_1,\ldots,x_k\in\Xi_\lambda\), is a locally finite fibre process in \(\Reals^d\). 
 That is to say, for any compact set \(K\) the total length of \(\mathcal{N}(\Xi_\lambda)\cap K\) is almost surely finite.
 \item\label{def:SIRSN-item-finite-intensity} 
 The length intensity \(\ell\) of \(\mathcal{N}(\Xi_1)\) (the mean length per unit area) is finite.
 \item\label{def:SIRSN-item-SIRSN} 
Suppose the Poisson point processes \(\{\Xi_\lambda:\lambda>0\}\) are coupled so that \(\Xi_{\lambda_1}\supseteq\Xi_{\lambda_2}\) if \(\lambda_1<\lambda_2\). The fibre process
 \[
\bigcup_{\lambda>0}\bigcup_{x_1,x_2\in\Xi_\lambda} \left( \mathcal{R}(x_1,x_2)\setminus(\ball(x_1,1)\cup\ball(x_2,1)) \right)
 \]
 has length intensity bounded above by a finite constant \(p(1)\).
\end{enumerate}
If only properties \ref{def:SIRSN-item-route}-\ref{def:SIRSN-item-finite-intensity} are satisfied, then the random mechanism is called a \emph{weak SIRSN}.
If only properties \ref{def:SIRSN-item-route}-\ref{def:SIRSN-item-locally-finite} are satisfied, then the random mechanism is called a \emph{pre-SIRSN}.
\end{defn}

\cite{AldousGanesan-2013} describe the \emph{binary hierarchy model},
a structure for providing planar routes, based on minimum-time paths using a dyadic grid furnished with speeds and uniformly randomized in orientation and position.
\cite{Aldous-2012} proves that this is a full planar SIRSN satisfying all the requirements of Definition \ref{def:SIRSN}. 
\cite{AldousGanesan-2013} also propose two other candidates for planar SIRSNs which do not involve the somewhat unnatural randomization required for the binary hierarchy model: 
one is based on route-provision \emph{via} a scale-invariant improper Poisson line process marked with random speeds
(the \emph{Poisson line process model}); and the other uses a dynamic proximity graph related to the Gabriel graph.
The purpose of the present paper is to explore the Poisson line process model: we will show that it is at least a pre-SIRSN if \(d=2\), and moreover we will show that even in dimension \(d>2\)
the construction provides a random metric space on \(\Reals^d\) (in particular it satisfies at least properties \ref{def:SIRSN-item-route}-\ref{def:SIRSN-item-invariance}
of Definition \ref{def:SIRSN}, with the possible exception of uniqueness of route). 
This therefore establishes the significance of the Poisson line process model as a scale-invariant random spatial network, while leaving open the question of whether it is a weak SIRSN or
even a full SIRSN, not just a pre-SIRSN.

The chief difficulty in analyzing any of these random mechanisms lies in the fact that it is hard to work with explicit minimum-time paths, whose explicit construction would involve solving a non-local minimization problem to determine geodesics.
\citet{AldousKendall-2007} and \citet{Kendall-2011b,Kendall-2014a} use approximations known as ``near-geodesics'', constructed using a kind of greedy algorithm. 
\citet{BaccelliTchoumatchenkoZuyev-2000} and \citet{BroutinDevillerHemsley-2014} study Delaunay tessellation paths that are determined using either their relationship to appropriate Euclidean straight lines or the so-called ``cone walk''.
\citet{LaGatta-2011} studies geodesics determined by random \emph{smooth} Riemannian structures, for which conventional calculus methods are available.
In the following, we argue for existence of minimum-time paths by exploiting properties of a 
Sobolev space of paths, and then by using indirect arguments.

The structure of the paper is as follows. 
The rest of this introduction (Section \ref{sec:prelims}) is concerned with basic notions of stochastic geometry (Subsection \ref{sec:notation})
and with the definition of the underlying improper Poisson line process \(\Pi\) marked with speeds (Subsection \ref{sec:improper}). 
This improper Poisson line process \(\Pi\) is defined by an intensity measure \((\gamma-1) v^{-\gamma}\d{v}\,\mu_d(\d{\ell})\) (for speed \(v>0\), parameter \(\gamma>1\), and invariant measure \(\mu_d\) on line-space)
and supplies a measurable orientation field marked by speeds:
Section \ref{sec:lipschitz-paths} then explores the way in which the measurable orientation field can be integrated to provide Lipschitz paths based on the marked line process, namely \(\Pi\)-paths.
Sobolev space and comparison arguments can then be used to establish \emph{a priori} bounds on Lipschitz constants for finite-time \(\Pi\)-paths (Theorem \ref{thm:a-priori-bound}),
hence closure, weak closure, and finally weak compactness (Corollary \ref{cor:compactness}) of finite-time \(\Pi\)-paths. All these results require \(\gamma\geq d\).
Note that dimension \(d>1\) if line-process theory is to be non-vacuous.

Section \ref{sec:pi-paths and points} shows that, given \(\gamma>d\) and fixed points \(x_1,x_2\in\Reals^d\), it is almost surely possible to connect \(x_1\) to \(x_2\) in finite time with \(\Pi\)-paths 
(Theorem \ref{thm:connection}), and indeed with probability \(1\) it is possible to connect \emph{all} pairs of points in this way (Theorem \ref{thm:metric-space}). 
Combined with Corollary \ref{cor:compactness}, this implies the existence of minimum-time \(\Pi\)-paths, namely \emph{\(\Pi\)-geodesics} (Definition \ref{def:geodesic}, Corollary \ref{cor:geodesics}).
In dimension \(d>2\) this is a rather unexpected result, since almost surely none of the lines of \(\Pi\) will then intersect. 
Nevertheless, \(\Pi\) then furnishes \(\Reals^d\) with the structure of a random geodesic metric space. 
In these higher dimensions it is difficult to imagine what a \(\Pi\)-geodesic might look like (Figure \ref{fig:dumbbell} illustrates the easier \(d=2\) case): 
however Definition \ref{def:near-sequential-pi-path}, Theorem \ref{thm:pi-path-approximation} and Corollary \ref{cor:pi-path-approximation} describe a class of ``\(\eps\)-near-sequential-\(\Pi\)-paths''
which can be used to approximate (and to simulate) \(\Pi\)-geodesics (Theorem \ref{thm:near-sequential-compactness}). 
In particular these results imply measurability of the random time taken to pass from one point to another using a \(\Pi\)-geodesic (Corollary \ref{cor:measurable-time}).

The remainder of the paper is restricted to the planar case of \(d=2\), since the arguments now make essential use of point-line duality.
Consider the extent to which networks formed by \(\Pi\)-geodesics fulfil the requirements of Definition \ref{def:SIRSN}. 
The statistical invariance property \ref{def:SIRSN-item-invariance} follows immediately from similar invariance of the underlying intensity measure of the improper Poisson line process
(whether planar or not).
Property \ref{def:SIRSN-item-route} requires almost sure uniqueness of network routes: 
Section \ref{sec:pi-geodesics-uniqueness} establishes this for \(\gamma>d=2\) (Theorem \ref{thm:uniqueness}),
using a careful analysis of the nature of planar \(\Pi\)-geodesics (Theorem \ref{thm:encounters}) which falls just short of establishing that planar \(\Pi\)-geodesics can be made up of consecutive sequences of line segments.
While \(\Pi\)-geodesics between pairs of points are minimum-time paths, the fact that they have finite mean length is not immediately apparent;
this is established in Section \ref{sec:pi-geodesics-finite-mean-length}, first for restricted planar \(\Pi\)-geodesics (Lemma \ref{lem:finite-mean-length-in-ball}),
then for general planar \(\Pi\)-geodesics (Theorem \ref{thm:finiteness-of-mean}).
Thus the finite-mean-length property \ref{def:SIRSN-item-finite-length} of Definition \ref{def:SIRSN} is verified for \(d=2\).
Finally the pre-SIRSN property \ref{def:SIRSN-item-SIRSN} is established for the planar case in Theorem \ref{thm:pre-SIRSN} of Section \ref{sec:properties};
here also is established the uniqueness of planar \(\Pi\)-geodesics reaching out to infinity (Theorem \ref{thm:unique-to-infinity}) and, 
for any specified point \(x\in\Reals^2\), 
the fact that all \(\Pi\)-geodesics emanating from \(x\) must coincide for initial periods (Theorem \ref{thm:coalescence-of-geodesics}). 
These results are established using an essentially soft argument concerning the existence of certain structures in \(\Pi\) (Lemma \ref{lem:pre-SIRSN-structure});
the concluding Section \ref{sec:conclusion} notes that more quantitative arguments would be required to decide whether the weak SIRSN or full SIRSN properties hold.
Section \ref{sec:conclusion} also notes some other interesting open questions.
\subsection[Notation and basic results]{Notation and basic results for random line processes}\label{sec:notation}
Random line processes (random patterns of lines) play a fundamental r\^ole in this study.
Here we review notation and basic results for un-sensed random line processes in Euclidean space, as described in \citet[Chapter 8]{ChiuStoyanKendallMecke-2013}.
(By an ``un-sensed line'', we mean a line without preferred direction.)
The corresponding theory for sensed lines follows from the observation that the space of sensed lines forms a double cover of the space of un-sensed lines.

Consider \emph{line-space}, the space \(\Lines^d\) of all un-sensed lines in \(\Reals^d\), for dimension \(d\geq2\). 
In the planar case \(d=2\) there is a natural geometric representation of \(\Lines^2\) as a punctured projective plane, 
since there is a \(3\)-space construction of the family of planar lines as the family of intersections of \(2\)-subspaces with a reference plane (say \(x_3=1\)).
More visually, but less naturally, \(\Lines^2\) can be viewed as a M\"obius band of infinite width.
Similar but less graphic geometric descriptions of \(\Lines^d\) (and its sensed counterpart) can be given in higher dimensional cases (\(d>2\)):
for example, the space of \emph{sensed} lines in \(\Reals^d\) can be represented using the standard immersion of the tangent bundle \(TS^{d-1}\) of the \((d-1)\)-sphere in \(\Reals^d\).

It is convenient to introduce notation for hitting events and hitting sets. For a line \(\ell\in\Lines^d\) and for \(K\) a compact subset of \(\Reals^d\), we write
\begin{equation}\label{eqn:def-hit}
 \ell \hits K
\end{equation}
for the statement that \(\ell\) intersects \(K\). 
We also introduce the \emph{hitting set} of \(K\) (the set of lines that hit \(K\)):
\begin{equation}\label{eqn:hitting-set}
 \hittingset{K} \quad=\quad 
\left\{\ell\in\Lines^d\;:\; \ell\hits K\right\}\,.
\end{equation}

General arguments show that there exists a measure on \(\Lines^d\) that is invariant under Euclidean isometries and unique up to a scaling factor.
Line-space \(\Lines^d\) can be constructed as the quotient space of the group of \(d\)-dimensional rigid motions by the subgroup that leaves a specified line invariant.
The existence of invariant measure on line-space
follows from the study of quotient measures for locally compact topological groups; a conceptual and general treatment of existence and uniqueness
is given by \citet[Section 2.3]{AbbaspourMoskowitz-2007} (see also \citealp[pp.~130-133]{Loomis-1953}), and follows here from unimodularity of the two groups in question.
\citet[Chapter 10]{Santalo-1976} and \citet{Ambartzumian-1990} describe alternative approaches that are direct but are computational rather than conceptual.
\begin{defn}\label{def:invariant-line-measure}
\emph{Invariant line measure} \(\mu_d(\d{\ell})\) is the unique measure on \(\Lines^d\) that is invariant under Euclidean isometries and is normalized by
the following requirement:
for all compact convex sets \(K\subset\Reals^d\) of non-empty interior (``convex bodies''),
the \(\mu_d\)-measure of the hitting set \(\hittingset{K}\) is half the Hausdorff \((d-1)\)-dimensional measure of the boundary of \(K\):
\begin{equation}\label{eqn:normalized-measure}
 \mu_d(\hittingset{K})\quad=\quad\half m_{d-1}(\partial K)\,.
\end{equation}
\end{defn}%
Here and in the following, \(m_{d-1}\) denotes Hausdorff measure of dimension \(d-1\).
The purpose of the normalization factor \(\tfrac12\) is to ensure that the \(\mu_d\)-measure of the hitting set of a fragment \(A\) of a flat hyper-surface is equal to 
its hyper-surface area \(m_{d-1}(A)\).

In the important special case of \(d=2\),
we can parametrize an un-sensed line \(\ell\in\Lines^2\) by (a) the angle \(\theta=\theta(\ell)\in[0,\pi)\) that it makes with a reference line (say, the \(x\)-axis),
and (b) the \emph{signed} distance \(r=r(\ell)\) between the line \(\ell\) and a reference point 
(conventionally taken to belong to the reference line; say, the origin \(\origin=(0,0)\)).  
Equation \eqref{eqn:normalized-measure} then takes a more explicit form:
\begin{equation}\label{eqn:planar-normalized-measure1}
 \mu_2(\d{\ell})\quad=\quad\half \;\d{r}\,\d{\theta}\,.
\end{equation}
More generally, the line measure \(\mu_d(\d{\ell})\) can be disintegrated using \((d-1)\)-dimensional Hausdorff measure on the hyperplane perpendicular to \(\ell\).
Let \(\varpi\) be the un-sensed direction of \(\ell\) and let \(y\) be its point of intersection on the perpendicular hyperplane.
Let \(\kappa_{s}=\tfrac{s/2}{\Gamma(1+s/2)}\) denote the \(s\)-dimensional volume of the unit ball in \(\Reals^s\), and for later convenience let \(\omega_{s-1}=s\kappa_s\)
denote the hyper-surface area of its boundary. Then
\begin{equation}\label{eqn:disintegration1}
 \mu_d(\d\ell) \quad=\quad \frac{1}{\kappa_{d-1}} \;  m_{d-1}(\d y)\, m_{S_+^{d-1}}(\d\varpi)\,,
\end{equation}
where the measure \(m_{S_+^{d-1}}\) is defined on the space of un-sensed line directions and can be thought of as \((d-1)\)-dimensional Hausdorff measure on the unit hemisphere \(S_+^{d-1}\)
in \(\Reals^d\). Proper interpretation of the representation \eqref{eqn:disintegration1} requires the space of un-sensed directions to
be considered as a further projective space, 
and the product measure to be twisted to take account of the fact that \(m_{d-1}\) here is defined on the hyperplane normal to the un-sensed direction of the line in question.
However the resulting discrepancies are confined to a null-set which can be ignored when considering invariant Poisson line processes.

An alternative representation, useful for certain calculations,
describes \(\mu_d\) in terms of the intersection of \(\ell\) with a fixed reference hyperplane. In two dimensions we obtain
\begin{equation}\label{eqn:planar-normalized-measure2}
 \mu_2(\d{\ell})\quad=\quad\half \, \sin\theta \; \d{p}\, \d{\theta}\,,
\end{equation}
where \(p=p(\ell)\) is the signed distance from the reference point \(\origin\) to the intersection of \(\ell\) with the reference line.
This alternative representation is defective: if \(\theta=0\) then there is no intersection and so \(p\) is ill-defined. 
However once again the resulting discrepancies are confined to a null-set which can be ignored when considering invariant Poisson line processes.
In higher dimensions the corresponding representation is
\begin{equation}\label{eqn:disintegration2}
 \mu_d(\d\ell) \quad=\quad \frac{\sin\theta}{\kappa_{{d-1}}} \;m_{d-1}(\d z)\, m_{S_+^{d-1}}(\d\varpi) \,,
\end{equation}
where \(\theta\) is the angle made by the un-sensed direction \(\varpi\) of the line \(\ell\) with the fixed reference hyperplane, and \(z\) locates the intersection of \(\ell\) with the reference hyperplane.

Note finally that the arguments of this paper depend only on the general forms of Equations (\ref{eqn:planar-normalized-measure1}, \ref{eqn:disintegration1}, \ref{eqn:planar-normalized-measure2}, \ref{eqn:disintegration2});
the exact constants involved
are not crucial.
\subsection{Improper Poisson line processes}\label{sec:improper}
Our constructions use \emph{Poisson} line processes. A unit-intensity Poisson line process in \(\Reals^d\) is obtained simply
by generating a Poisson point process on the corresponding representing space \(\Lines^d\) using the invariant measure
\(\mu_d\). It is a geometric consequence of the \(\sigma\)-finiteness of \(\mu_d\) that the resulting random line pattern is locally finite: only finitely many lines hit any given compact set.
However our constructions will use \emph{improper} Poisson line processes, which can be viewed as superpositions of infinitely many independent Poisson line processes, 
different line processes being thought of as representing highways with speed limits lying in different ranges.
If we augment the representation space by a mark space \((0,\infty)\) of speed-limits, then the improper Poisson
line process can be represented as a Poisson point process on \(\Lines^d\times(0,\infty)\), with a \(\sigma\)-finite intensity measure on \(\Lines^d\times(0,\infty)\) which is invariant under rigid motions but which 
does not project down onto a \(\sigma\)-finite intensity measure on \(\Lines^d\).
Thus the main actors in this account are invariant {improper} un-sensed Poisson line processes, 
with each line \(\ell\) being marked by a different positive speed-limit \(v=v(\ell)>0\). 
Scaling arguments \citep{Aldous-2012,AldousGanesan-2013} lead to a natural family of intensity measures
for such a marked line process, based on
a positive parameter \(\gamma>1\):
\begin{equation}\label{eqn:improper}
 (\gamma-1) \, v^{-\gamma} \,\d{v} \;\mu_d(\d{\ell})\,.
\end{equation}
The factor \(\gamma-1\) ensures that for all \(\gamma>1\) the sub-process of lines with marks \(v>1\) forms a unit-intensity Poisson line process which is of \emph{unit intensity}, in the sense that its mean intensity is the invariant measure given in \eqref{eqn:normalized-measure}, so that the mean number of lines hitting a flat fragment of hyper-surface is equal to its hyper-surface area.
In case \(d=2\) we may write this intensity measure as \(\half (\gamma-1) v^{-\gamma}\,\d{v}\,\d{r}\,\d{\theta}\). 
Fixing a general dimension \(d\) and parameter \(\gamma>1\), let \(\Pi=\Pi^{(d,\gamma)}\) denote the resulting random process of marked lines \((\ell, v(\ell))\). 
In the following, the dependence on \(d\) and \(\gamma\) will be clear from the context,
and consequently will be suppressed. Figure \ref{fig:dumbbell} illustrates the formation of minimum-time routes between two fixed collections of nodes, for varying values of the parameter \(\gamma>2\).
Note that spatial networks formed in this way will automatically satisfy
property \ref{def:SIRSN-item-invariance} of Definition \ref{def:SIRSN}, because of the invariance properties of the intensity measure \eqref{eqn:improper}.
\begin{Figure}
  \centering
  \includegraphics[width=1.5in]{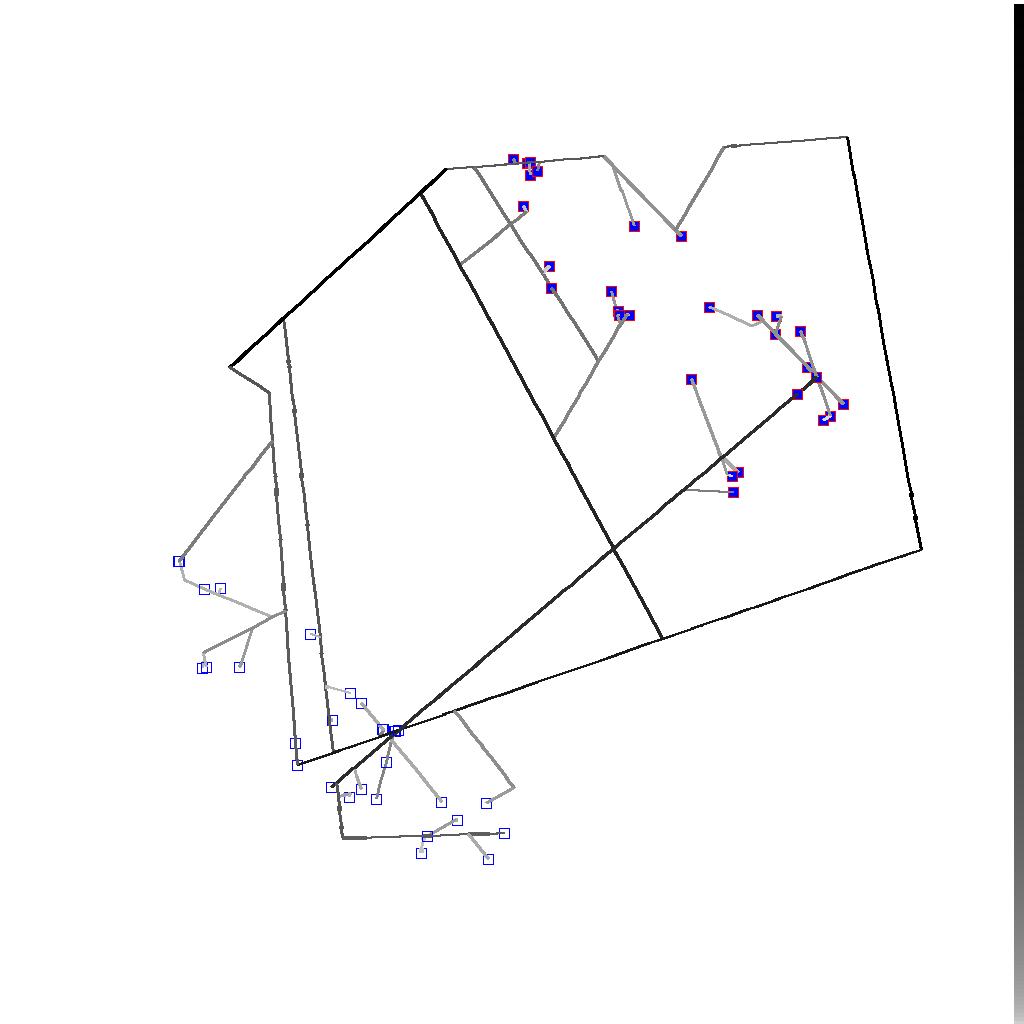}\hfil
  \includegraphics[width=1.5in]{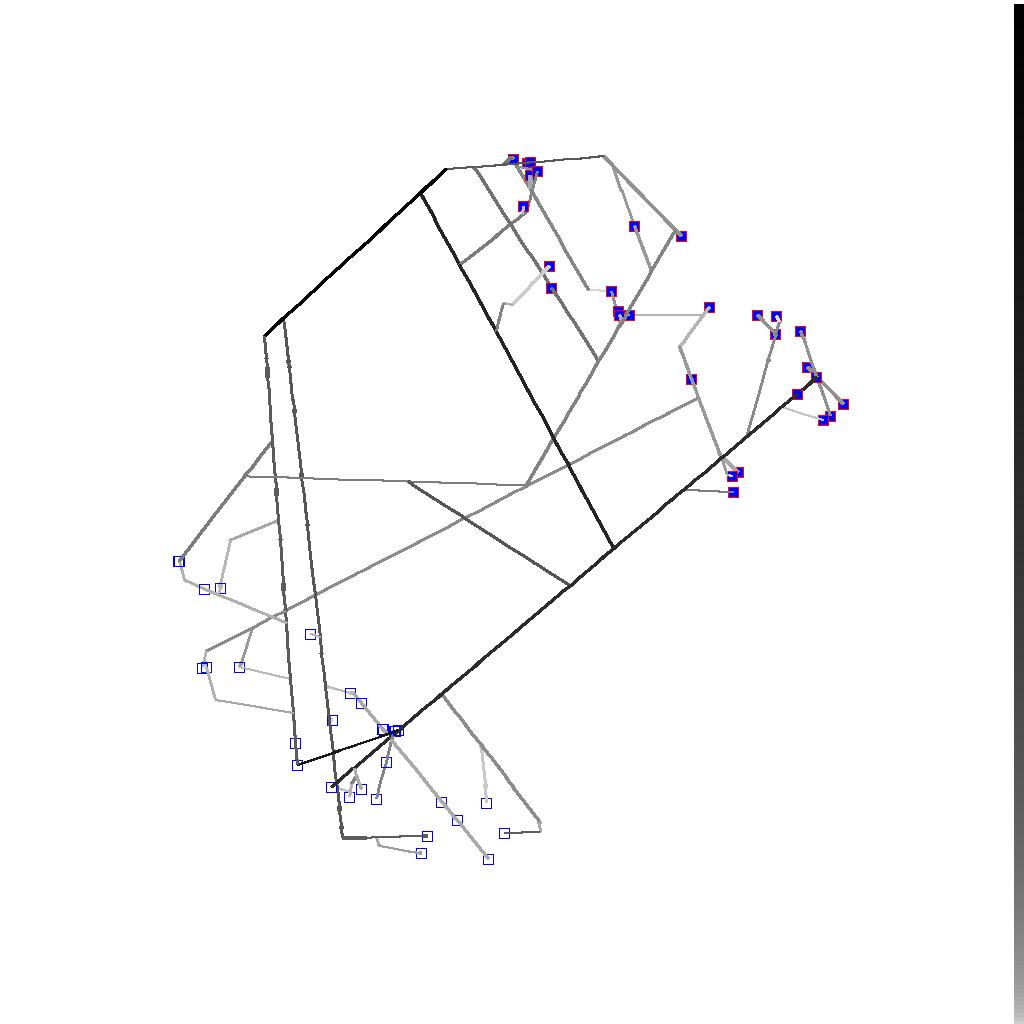}\hfil
  \includegraphics[width=1.5in]{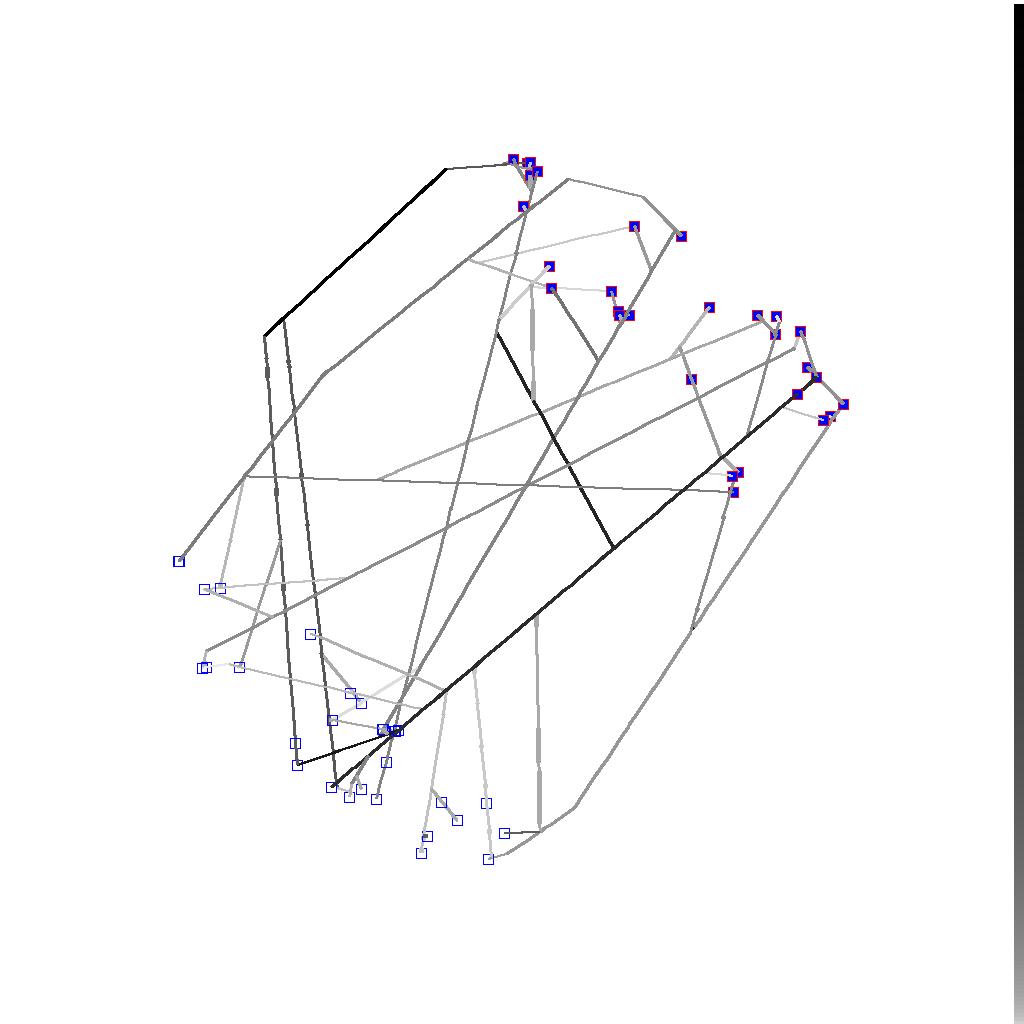}\hfil
  \includegraphics[width=1.5in]{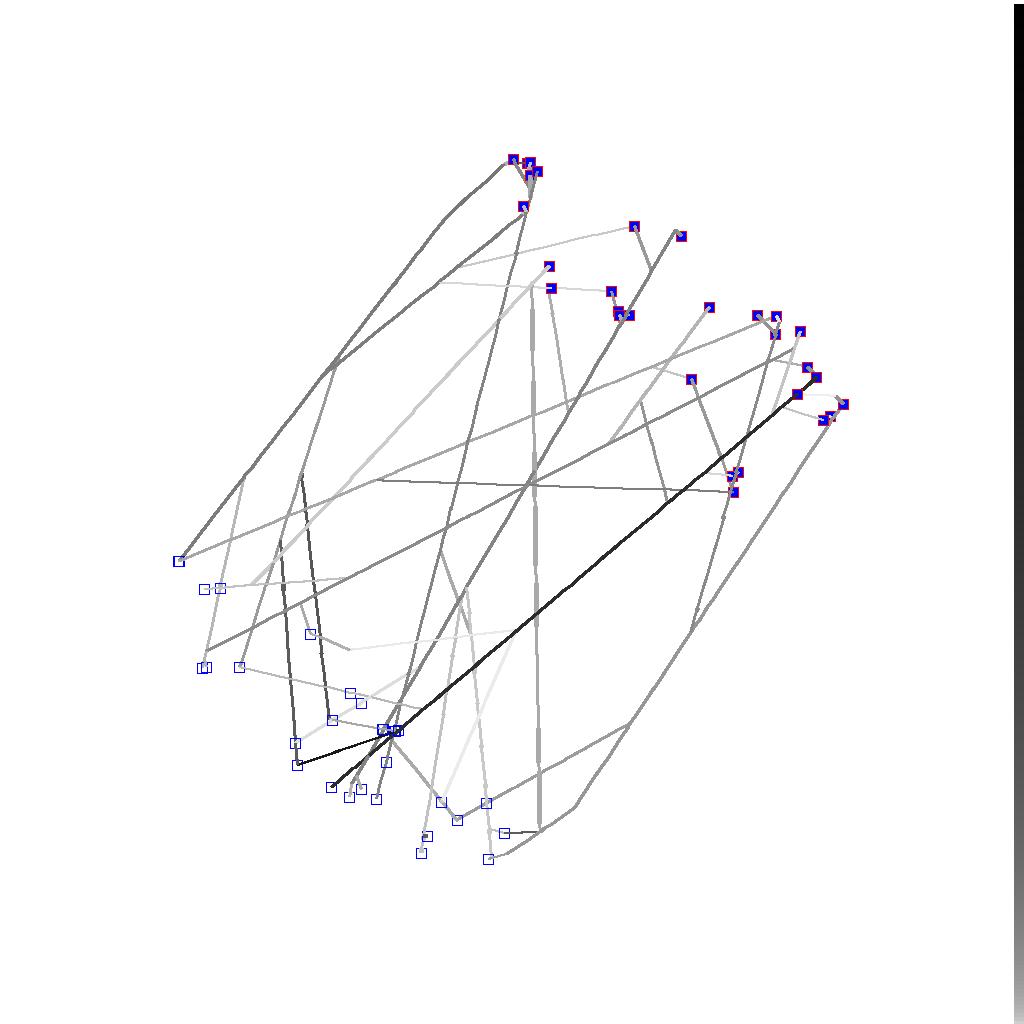}
\caption{\label{fig:dumbbell}
Minimum-time routes between two separated collections of nodes for networks built from the improper Poisson line process described in Section \ref{sec:improper},
for parameter \(\gamma\) taking values \(2.1,4.0,8.0,16.0\). Lighter segments have lower speed-limits.
Note that as \(\gamma\) increases so the routes become more direct, but also there is less route-sharing.
}
\end{Figure}

The intensity measure gives infinite mass to the set of lines intersecting any convex body, and therefore the union
 of all lines from \(\Pi\) is \emph{not}
a random closed set.
Consequently it is not possible to make direct application of the classic theory of random closed sets (as surveyed, for example, in \citealp[Chapter 6]{ChiuStoyanKendallMecke-2013}).
Indeed the union of all lines from \(\Pi\)
is almost surely everywhere dense, and the theory for such sets is obscure (see for example \citealp{AldousBarlow-1981,Kendall-2000b}).
We therefore focus on sub-patterns of lines subject to a positive lower bound on their speeds.
Consider the set of lines hitting a convex body \(K\) and having speed-limits no slower than \(v_0>0\):
\[
 \left\{ (\ell, v)\;:\; \ell\in\hittingset{K} \text{ and } v\geq v_0\right\}\,.
\]
This has finite intensity measure, since \(\gamma>1\). 
It follows that the full set of marked lines \(\left\{ (\ell,v) : \ell\in\hittingset{K}\right\}\) can be expressed as
a countable union of random closed sets (indeed, locally finite unions of lines) when broken up according to different ranges of speed-limit.
Hence the union of all these lines does have a natural representation as a random \(F_\sigma\).
Indeed it can be related to random closed set theory as follows.

\begin{defn}\label{def:silhouette}
For given \(d\geq2\) and \(\gamma>1\), and fixed \(v_0>0\), 
let \(\Pi_{v_0}\) denote the \emph{proper} marked Poisson line process of all lines with speed-limits no slower than \(v_0\):
\[
 \Pi_{v_0}\quad=\quad \left\{(\ell,v)\in\Pi\;:\; v\geq v_0\right\}\,.
\]
The \emph{silhouette} \(\Silhouette_{v_0}\) of \(\Pi_{v_0}\) is the random closed set which is the union of all lines in \(\Pi_{v_0}\):
\begin{equation}\label{eqn:silhouette}
 \Silhouette_{v_0}\quad=\quad 
\bigcup\left\{\ell\;:\;(\ell, v)\in\Pi_{v_0}\right\}
\quad=\quad\bigcup\left\{\ell\;:\;(\ell,v)\in\Pi \text{ and } v\geq v_0\right\}\,.
\end{equation}
So \(\Silhouette_{0+}=\bigcup_{v_0>0}\Silhouette_{v_0}\) can be viewed as a random \(F_\sigma\).
\end{defn}
Note that almost surely the \emph{unmarked} line process \(\{\ell:(\ell,v)\in\Pi_{v_0}\}\) can be recovered from the silhouette \(\Silhouette_{v_0}\).
Moreover we can recover \(\Pi\) in entirety from 
the details of the changes in \(\Silhouette_v\) as \(v\) varies, since \(\Silhouette_v\setminus\Silhouette_{v+}\) is the locally finite union of
lines whose speed-limits are exactly equal to \(v\).
Indeed \(\Pi_v =\Pi_{v_-}=\bigcap_{v<v_0}\Pi_v\) for \(v_0>0\), and \(\Pi_v\) changes only at countably 
many values of \(v>0\), and, almost surely, for all \(v>0\) if \(\Silhouette_v\setminus\Silhouette_{v+}\) is non-empty then it is composed of just one line.
Thus
\[
\Pi\quad=\quad
\left\{(v,\ell)\;:\;\ell= \Silhouette_v\setminus\Silhouette_{v_+} \text{ for some } v>0\right\}\,.
\]

For notational convenience we introduce
the \emph{maximum speed-limit function} \(V\) holding everywhere on \(\Reals^d\) and imposed by \(\Pi\).
This is a random upper-semicontinuous function
 \(V:\Reals^d\to[0,\infty)\) defined in terms of its upper level sets: 
\begin{equation}\label{eqn:speed-limit}
 \left\{x\;:\; V(x)\geq v_0\right\} \quad=\quad \Silhouette_{v_0}
\quad=\quad\bigcup\left\{\ell\;:\;(\ell,v)\in\Pi_{v_0}\right\}\,.
\end{equation}
As with random dense line patterns, there is no satisfactory theory for general random upper-semicontinuous functions: we
use \(V\) merely as a convenient mathematical short-hand for the filtration of random closed sets \(\{\Silhouette_v:v>0\}\).

\section[Lipschitz paths]{Lipschitz paths and networks}\label{sec:lipschitz-paths}
This section introduces the notion of \(\Pi\)-paths; locally Lipschitz paths traversing a system of ``roads'' (supplied by  \(\Pi\)) furnished with speed-limits \emph{via} the maximum speed-limit function \(V\).
We will formulate this notion carefully and prove results yielding a variational context within which to study the minimum-time \(\Pi\)-paths (``\(\Pi\)=geodesics'') between specified points.
Care is needed, because we cannot assume that such paths are built up using consecutive sequences of intervals spent on different roads (and indeed this absolutely cannot be the case for dimension \(d>2\)). 
Instead the \(\Pi\)-paths are best viewed using such intervals arranged in a tree pattern rather than ordered sequentially (compare the use of trees to represent bounded variation paths in \citealp{HamblyLyons-2010}).

From henceforth we shall fix a dimension \(d\geq2\) (since the case \(d=1\) is trivial) and a parameter \(\gamma>1\) 
(note however that the discussion 
of this section will lead to imposition of progressively more severe
constraints on \(\gamma\)). 
Recall (for example, from \citealp[ch.5]{Evans-1998}) that a Lipschitz curve \(\xi:[0,T)\to\Reals^d\) satisfies \(|\xi(s)-\xi(t)|\leq A|s-t|\) when \(0\leq s<t<T\),
for some constant \(A>0\).
The least such \(A\) is the \emph{Lipschitz constant} \(\text{Lip}(\xi)\). 
The Lipschitz property for \(\xi\) using constant \(A\) holds if and only if \(\xi\) is absolutely continuous with almost-everywhere defined weak derivative \(\xi'\), with \(\operatorname{ess\ sup}_t|\xi'(t)|\leq A\).

We first discuss two preparatory results; a simple
lemma relating the velocity of a general Lipschitz path (defined for almost all time) to the directions of lines which it visits, 
and a corollary concerning the way in which such a Lipschitz path behaves
at the intersections formed by a pattern of lines. Intuitively speaking, if the path velocity has non-zero component normal to a given line 
then it must move away immediately, so for almost all time either the path 
runs on the line and has velocity parallel to the line, or the path does not lie on the line at all.
\begin{lemma}\label{lem:direction}
 Let \(\xi:[0,T)\to\Reals^d\) be a locally Lipschitz path and let \(\ell\) be a line, both lying in \(d\)-dimensional space \(\Reals^d\).
Suppose that \(\underline{e}\) is a unit vector parallel to the direction of \(\ell\).
Then the time-set
\[
\left\{t\in[0,T)\;:\; \xi(t)\in\ell \text{ and }\xi^\prime(t)\neq \langle\xi^\prime(t),\underline{e}\rangle\underline{e}\;\right\}                           
\]
is a Lebesgue-null subset of \([0,T)\).
\end{lemma}
\begin{proof}
Let \(P\) denote projection onto the hyperplane normal to \(\ell\).
The line \(\ell\) projects under \(P\) to a singleton point which we denote by \(P\ell\).
Restricting to compact subsets of \([0,T)\) if necessary, it suffices to treat the case in which \(\xi\) is globally Lipschitz over \([0,T)\); let \(\alpha=\text{Lip}(\xi)\) be the corresponding Lipschitz constant, so that
\[
 |\xi^\prime(t)| \quad\leq\quad \alpha \qquad \text{for almost all }t\in[0,T)\,.
\]
The set \(\{t\in[0,T):P\xi(t)=P\ell \text{ and }P\xi^\prime(t)\neq0\}\) is the countable union of time-sets
\[
 A^\pm_{j,n} \quad=\quad \left\{t\in[0,T):P\xi(t)=P\ell \text{ and }\pm\langle P(\xi^\prime)(t), \underline{e}_j\;\rangle>\tfrac1n\right\}\,,
 \quad\text{for } j=1, \ldots, d-1\,,
\]
where \(\underline{e}_1, \ldots, \underline{e}_{d-1}\) are orthogonal unit vectors perpendicular to \(\ell\), 
so it suffices to show each \(A^\pm_{j,n}\) is Lebesgue-null (note that \(P\xi^\prime\) is not continuous, so that \(A^\pm_{j,n}\) may not be open). 

Without loss of generality, fix attention on \(A^+_{1,n}\); we will complete the proof by showing that this is Lebesgue-null.
Fix arbitrary \(\eps>0\) and cover \(A^+_{1,n}\) by a disjoint countable union of closed bounded intervals \(F=\bigcup_i [a_i,b_i]\), such that
\begin{equation}\label{eq:residue}
  \Leb(F\setminus A^+_{1,n})\quad<\quad \frac{\eps}{\alpha n}\,.
\end{equation}
Since \(\xi\) is continuous, we may shrink each interval \([a_i,b_i]\) so as to ensure that \(P\xi=P\ell\) at \(a_i\) and \(b_i\), while still preserving \eqref{eq:residue} and the 
covering property \(A^+_{1,n}\subseteq F\). Since \(P\xi=P\ell\) on the end-points of each \([a_i,b_i]\),
\[
\int_{F\setminus A^+_{1,n}} P\xi^\prime(t)\d{t} + \int_{A^+_{1,n}} P\xi^\prime(t)\d{t}
 \quad=\quad \int_F P\xi^\prime(t)\d{t}\quad=\quad \sum_i (P\xi(b_i)-P\xi(a_i))\quad=\quad 0\,.
\]
However we can apply the Lipschitz property of \(\xi\) to show that
\begin{equation*}
\alpha \Leb(F\setminus A^+_{1,n})
\quad\geq\quad \left|\int_{F\setminus A^+_{1,n}} P\xi^\prime(t)\d{t}\right| 
\quad\geq\quad \left|\int_{A^+_{1,n}} \langle P\xi^\prime(t),\underline{e}_1\rangle\d{t}\right|\\
\quad\geq\quad  \frac{1}{n}\Leb(A^+_{1,n})
\end{equation*}
and therefore
\[
 \Leb(A^+_{1,n}) \quad\leq\quad \alpha \, n \Leb(F\setminus A^+_{1,n})\quad\leq\quad \eps\,.
\]
Since \(\eps>0\) is arbitrary, the result follows.
\end{proof}
 As a direct consequence of Lemma \ref{lem:direction}, the only way in which a Lipschitz path can spend positive time on the intersection 
of two distinct lines is by being at rest.
\begin{cor}\label{cor:network-intersections}
 Consider a network in \(\Reals^d\) formed by the union of a countable number of distinct lines \(\ell_1\), \(\ell_2\), \ldots, and form the
\emph{intersection point pattern}
\[
 \mathcal{I} \quad=\quad \bigcup_{1\leq i < j < \infty} \ell_i\cap\ell_j\,.
\]
If \(\xi:[0,T)\to\Reals^d\) is a locally Lipschitz curve in \(\Reals^d\) then 
the time-set
\begin{equation}\label{eqn:nullity}
\left\{t\in[0,T)\;:\; \xi(t)\in\mathcal{I} \text{ and } |\xi^\prime(t)|>0\right\} 
\end{equation}
must be a Lebesgue-null subset of \([0,T)\).                                                                                                                                                
\end{cor}
\begin{proof}
 Since \(\mathcal{I}\) is a countable union of points, it suffices to consider two distinct lines \(\ell_1\) and \(\ell_2\)
with non-empty intersection. 
Let \(\underline{e}_i\) be a unit vector parallel to the direction of \(\ell_i\). Note that, since the \(\ell_i\) are distinct and meet (and therefore cannot be parallel), it follows that 
the unit vectors \(\underline{e}_1\) and \(\underline{e}_2\) must be linearly independent.
By Lemma \ref{lem:direction}
the following two time-sets are both Lebesgue-null:
\begin{eqnarray*}
&\left\{t\in[0,T)\;:\; \xi(t)\in\ell_1 \text{ and } \xi^\prime(t) \neq \langle\xi^\prime(t),\underline{e}_1\rangle\underline{e}_1\;\right\}\,,\\
&\left\{t\in[0,T)\;:\; \xi(t)\in\ell_2 \text{ and } \xi^\prime(t) \neq \langle\xi^\prime(t),\underline{e}_2\rangle\underline{e}_2\;\right\}\,.
\end{eqnarray*}
If \(\xi^\prime\) is 
simultaneously parallel to the two
linearly independent \(\underline{e}_i\) then it must vanish. Consequently
the time-set defined by Equation \eqref{eqn:nullity} is a subset of the union of these two sets, and is therefore Lebesgue-null.
\end{proof}

We now define the fundamental notion explored in this paper.
\begin{defn}\label{def:pi-path}
 A \emph{\(\Pi\)-path} is a locally Lipschitz path in \(\Reals^d\) which for almost all time obeys the speed-limits imposed
by \(\Pi\) \emph{via} the random upper semicontinuous function \(V:\Reals_d\to[0,\infty)\). 
Expressed more precisely, the \(\Pi\)-path is given by an \(\Reals^d\)-valued function
\[
\xi:[0,T)\to\Reals^d\,
\]
defined up to a (possibly infinite) \emph{terminal time} \(T>0\), and satisfying the condition
that, for all \(v>0\), the time-set
\[
\left\{t\in[0,T)\;:\; |\xi^\prime(t)| > v \text{ and } \xi(t)\not\in\Silhouette_v\right\}                                           
\]
is a Lebesgue-null subset of \([0,\infty)\).
Let \(\Paths_T\) be the space of all \(\Pi\)-paths defined up to time \(T>0\), and let \(\Paths=\bigcup_{T>0}\Paths_T\) be the space of all \(\Pi\)-paths whatsoever.
\end{defn}
\begin{rem}\label{rem:pi-path}
Direct arguments using Lemma \ref{lem:direction} show that
the Lebesgue-null condition in Definition \ref{def:pi-path} can be replaced by any one of three equivalent conditions:
\begin{enumerate}
\item\label{def:pi-path-condition1} for almost all \(t\in[0,T)\) for which \(\xi^\prime(t)\neq0\), 
it is the case that \(\xi(t)\in\Silhouette_{|\xi^\prime(t)|}\);
\item\label{def:pi-path-condition2} for almost all \(t\in[0,T)\) for which \(\xi^\prime(t)\neq0\), 
it is the case that the line \(\xi(t)+\xi'(t)\Reals\) belongs to \(\Pi_{|\xi'(t)|}\);
\item\label{def:pi-path-condition3} for almost all \(t\in[0,T)\), 
it is the case that \(|\xi'(t)|\leq V(\xi(t))\).
\end{enumerate}
\end{rem}
Here 
\(\xi(t)+\xi'(t)\Reals\) denotes the line through \(\xi(t)\) with orientation \(\xi'(t)\).

Condition \ref{def:pi-path-condition1} implies that
for almost all \(t\in[0,T)\), if \(\xi(t)\not\in\bigcup_{v>0}\Silhouette_v\)
then \(\xi'(t)=0\). 


Crucially, \(\Pi\)-paths integrate the measurable orientation field provided by \(\Pi\) and obey the speed-limits imposed by \(\Pi\).
\begin{lem}\label{lem:pi-path-and-intersections}
 Suppose that \(\xi\) is a \(\Pi\)-path. For any \((\ell,v)\in\Pi\), let \(\underline{e}\) be a unit vector parallel to the direction of \(\ell\). Then
the following time-sets are both Lebesgue-null:
\begin{align}
 &\left\{t\in[0,T)\;:\; \xi(t)\in\ell, \xi'(t)\neq\langle\xi'(t),\underline{e}\rangle\underline{e}\;\right\}\,,\label{eq:pi-path-and-intersections1}\\
 &\left\{t\in[0,T)\;:\; \text{ for some }(\ell,v)\in\Pi, \xi(t)\in\ell, |\xi'(t)|>v\right\}\,.\label{eq:pi-path-and-intersections2}
\end{align}
\end{lem}
\begin{proof}
It is a direct consequence of Lemma \ref{lem:direction} that the time-set at \eqref{eq:pi-path-and-intersections1} is Lebesgue-null.

As for the time-set at \eqref{eq:pi-path-and-intersections2}, note that, for some Lebesgue null-set \(\mathcal{N}_1\), 
\begin{multline*}
 \left\{t\in[0,T)\;:\; |\xi'(t)|>v \text{ and } \xi(t)\in\ell \text{ for some }(\ell,v)\in\Pi\right\} \quad\subseteq\quad\\
\left\{t\in[0,T)\;:\; \xi(t)\in\ell \text{ and }\xi(t)\in\widetilde\ell \text{ for some } (\ell,v), (\widetilde\ell,w)\in\Pi\text{ with }w>v\right\}\cup \mathcal{N}_1
\end{multline*}
(use the equivalent \(\Pi\)-path definition \ref{def:pi-path}). But this implies that
\begin{multline*}
 \left\{t\in[0,T)\;:\; |\xi'(t)|>v \text{ and } \xi(t)\in\ell \text{ for some }(\ell,v)\in\Pi\right\} \\
\quad\subseteq\quad
\left\{t\in[0,T)\;:\; |\xi'(t)|>v \text{ and } \xi(t)\in\mathcal{I}\right\}\,,
\end{multline*}
where \(\mathcal{I}\) is the intersection point pattern derived from \(\Pi\) as in the statement of Corollary \ref{cor:network-intersections}.
By Corollary \ref{cor:network-intersections}, the time-set on the right-hand side
is Lebesgue-null.
\end{proof}

As noted at the start of this section, there are two related reasons for taking
this rather abstract approach to \(\Pi\)-paths, as opposed to working only with paths built up sequentially from segments of lines in \(\Pi\). Firstly, in dimension \(d>2\) there are no
non-trivial sequential \(\Pi\)-paths, since 
almost surely no two lines of \(\Pi\) will intersect. 
Secondly, even in the planar case of \(d=2\) we must consider
non-sequential \(\Pi\)-paths as possible
limits of simple \(\Pi\)-paths, for example when establishing the existence of minimizers of \(\Pi\)-path functionals (specifically, 
the functional yielding accumulated elapsed time).

Here and in the following, we establish a number of statements about \(\Paths_T\) and \(\Paths\), all of which should be qualified as holding ``almost-surely'', since they depend on the random construction of 
the marked line process \(\Pi\). Similarly ``constants'' are actually random variables measurable with respect to the line process \(\Pi\). 
Borrowing the terminology of Random Walks in Random Environments, it might be said that we are concerned with the quenched law based on the random environment provided by \(\Pi\).

It is immediate from the definition that
\(\Pi\)-paths are individually fully Lipschitz over time intervals in which the \(\Pi\)-path in question belongs to a specified compact set.
As a consequence, we can establish \emph{a priori} bounds on the length of any \(\Pi\)-path beginning in a specified compact set \(K\) and with terminal time at most \(T<\infty\)
(hence finite), so 
long as \(\gamma\geq d\). These bounds will depend on \(T<\infty\), \(K\), \(\gamma\), and also on the random marked line pattern \(\Pi\), and 
will allow us to deduce a uniform Lipschitz property for
all \(\Pi\)-paths from \(\Paths_T\) which start in \(K\). 
Moreover, if \(\gamma\geq d\) then the set \(\Paths_T\) is weakly closed, and therefore the family of 
all \(\Pi\)-paths from \(\Paths_T\) started in \(K\) is weakly compact. This allows us to make sense of the notion of \emph{\(\Pi\)-geodesics}, measuring path-length 
not as geometric length but using total time of passage along the path.
We first establish an \emph{a priori} upper bound for the Lipschitz constants of \(\Pi\)-paths begun near the origin and defined up to a finite time.

\begin{thm}[An \emph{a priori} bound for path space]\label{thm:a-priori-bound}
 Suppose that \(\gamma\geq d\geq2\). Fix \(T<\infty\) and \(r_o>0\), and consider a \(\Pi\)-path \(\xi\) with \(\xi(0)\in\ball(\origin,r_0)\).
Then there is \(C=C(\gamma, T, r_0, \Pi)<\infty\) (which is a constant when conditioned on the particular realization of the marked line process \(\Pi\)) such that almost surely
\[
 |\xi(t)| \quad\leq\quad C \qquad \text{for all }t\in[0,T)\,.
\]
\end{thm}
\begin{proof}
 We consider a given realization of the speed-limit function \(V\) (equivalently, of \(\Pi\)). It suffices to obtain a lower bound on the time at
which \(\xi\) first exits \(\ball(\origin,r)\) for a given \(r>r_0\), and to show that this bound tends to infinity as \(r\to\infty\).
We achieve this by constructing a comparison with a one-dimensional system (controlled by the speed-limit function \(V\)), thus delivering an upper bound on \(|\xi|\),
and then by showing that with probability \(1\) the underlying configuration of \(\Pi\) is such that
the comparison system takes infinite time to diverge to infinity.

Our comparison system \(y:[0,\infty)\to[0,\infty)\) satisfies
\begin{align}\label{eq:comparison-system}
 y'(t) \;&=\; \overline{V}(y(t)) \;=\; \sup\{V(x)\;:\; |x|\leq y(t)\}
\;=\; \sup\{v:(\ell,v)\in\Pi, \ell\hits\ball(\origin,r)\}\,,\\
 y(0)  \;&=\; r_0\,.\nonumber
\end{align}
So \(y\) could be thought of as the distance from \(\origin\) of an idealized path which always travels radially at the maximum speed \(\overline{V}\)
available from \(\Pi\) within that distance. Standard comparison techniques show that if \(\xi\) is a \(\Pi\)-path then
\[
 y(t) \quad\geq\quad |\xi(t)| \qquad \text{for almost all \(t\) whenever }|\xi(0)|\leq r_0\,.
\]

The next step is to control the growth of the maximum speed-limit \(\overline{V}(y)\) as a function of \(y\). 
We introduce a nonlinear projection from marked line-space to the quadrant \([0,\infty)^2\):
\[
 (\ell, v) \quad\mapsto\quad \left(|r|^{d-1}, v^{-(\gamma-1)}\right)\quad=\quad (p, s)\,.
\]
We think of \(p=|r|^{d-1}\) as ``generalized distance'', and \(s=v^{-(\gamma-1)}\) as ``meta-slowness''. The fibres of the projection have zero invariant line measure.
Bearing in mind that \(\gamma\geq d\geq2\), and using the expression for the intensity measure at \eqref{eqn:improper}, also for \(\mu_d\) at \eqref{eqn:disintegration1}, the image of \(\Pi\) 
under the projection is a stationary Poisson point process on the quadrant \([0,\infty)^2\), with intensity measure
\[
 (\gamma-1)\frac{\omega_{d-1}}{2} \times(-v^{-\gamma}\d{v}) \times\d(|r|^{d-1})
\quad=\quad
\frac{\omega_{d-1}}{2} \,\d(v^{-(\gamma-1)}) \,\d(|r|^{d-1})
\quad=\quad
\frac{\omega_{d-1}}{2} \,\d{p}\,\d{s}
\]
(recall that \(\omega_{d-1}\) is the hyper-surface area of the unit ball in \(\Reals^d\)).

This projection to a planar point process
delivers a Poisson point process \(\widetilde\Pi\) of constant intensity \(\tfrac12{\omega_{d-1}}\)
 in the quadrant. Note that
\[
 \left(\overline{V}(r)\right)^{-(\gamma-1)}\quad=\quad\inf\left\{s : (p,s) \in \widetilde\Pi \text{ and } p\leq r^{d-1}\right\}\,,
 \qquad \text{ for } r>0\,.
\]
Accordingly we can use \eqref{eq:comparison-system} to obtain information about the sequence of generalized distances at which the meta-slowness of the comparison system \(y\) changes,
deriving a recursive expression (compare the use of perpetuities in \citealp{Kendall-2011b,Kendall-2014a}). 
Initially the random meta-slowness of \(y\) at
generalized distance \(P_0=r_0^{d-1}\) is
\[
 S_0 \text{ of distribution } \text{Exponential}\left(\tfrac12{\omega_{d-1}} \, P_0\right)\,.
\]
Let \(P_0<P_1<P_2<\ldots\) be the sequence of generalized distances at which the meta-slowness changes through values \(S_0>S_1>S_2>\ldots\). 
Poisson process arguments show
\begin{align}\nonumber
 P_n-P_{n-1} \quad&=\quad
 \frac{1}{S_{n-1}} T_n \qquad \text{ for } T_n 
\text{ of distribution } \text{Exponential}\left(\tfrac12{\omega_{d-1}}\right)\,,\\
S_n \quad&=\quad 
 U_n S_{n-1} \qquad \text{ for } U_n 
\text{ of distribution } \text{Uniform}(0,1)\,.
\label{eqn:meta-slowness-recursion}
\end{align}
Moreover the \(T_i\)'s and \(U_i\)'s are all independent.

It is useful to re-organize this recursion
so as to recognize the product \(X_n=S_n P_n\)
as a Markov chain (in fact a perpetuity):
\begin{equation}\label{eq:perpetuity}
 X_n=S_n P_n \quad=\quad U_n (T_n + S_{n-1} P_{n-1})\quad=\quad U_n(T_n+X_{n-1})\,.
\end{equation}
Iteration shows that this perpetuity has a limiting distribution, expressible as an infinite 
almost-surely convergent sum
\[
 U_1 T_1 + U_1 U_2 T_2 + U_1U_2U_3T_3 + \ldots
\]
(stronger results on perpetuities can be found in
the foundational paper of \citealp{Vervaat-1979}). Indeed, Markov chain arguments show that \(X_n\) converges
to this distribution in total variation with geometric ergodicity \citep[Ch.~15 especially Theorem 15.0.1]{MeynTweedie-1993}. This follows by noting that the Markov chain
\(\{X_n:n\geq0\}\) is Lebesgue-irreducible and satisfies a geometric Foster-Lyapunov condition based on (for example)
\begin{multline*}
 \Expect{1+X_n \;|\; X_{n-1}}\quad=\quad
1 + \half\left(\frac{2}{\omega_{d-1}}+X_{n-1}\right)\\
\quad\leq\quad
\begin{cases}
 \frac{2}{3} (1+X_{n-1}) & \text{ if } 1+X_{n-1} > 3\left(1+ \frac{2}{\omega_{d-1}}\right)\,,\\
 \frac{2}{3} (1+X_{n-1})+\frac12 +\frac{1}{\omega_{d-1}} & \text{ otherwise. }
\end{cases}
\end{multline*}
The interval \([0, 3(1 + \tfrac{2}{\omega_{d-1}}))\) is a small set for the Markov chain \(X\)
(since \(U_n\) and \(T_n\) are independent and
have continuous densities which are strictly positive over their ranges of \([0,1]\) and \([0,\infty)\)), 
so this is indeed a geometric Foster-Lyapunov condition, and establishes geometric ergodicity for the Markov chain \(X\).

Consider the elapsed actual time between generalized positions \(P_{n-1}\) and \(P_n\), namely
\[
 (S_{n-1})^{\frac{1}{\gamma-1}} \times \left(P_n^{\frac{1}{d-1}}-P_{n-1}^{\frac{1}{d-1}}\right)\,.
\]
Summing over \(n\), and using the requirement that \(\gamma\geq d\), the total time till \(y\) reaches infinity is given by the sum
\begin{multline*}
 \sum_n (S_{n-1})^{\tfrac{1}{\gamma-1}} \times \left(P_n^{\frac{1}{d-1}}-P_{n-1}^{\frac{1}{d-1}}\right) 
\;=\;
 \sum_n (S_{n-1})^{\tfrac{1}{\gamma-1}-\tfrac{1}{d-1}} \times \left((S_{n-1}P_n)^{\frac{1}{d-1}}-(S_{n-1}P_{n-1})^{\frac{1}{d-1}}\right)\\
\quad=\quad
 \sum_n (S_{n-1})^{\tfrac{1}{\gamma-1}-\tfrac{1}{d-1}} \times \left((S_{n-1}P_{n-1}+T_n)^{\frac{1}{d-1}}-(S_{n-1}P_{n-1})^{\frac{1}{d-1}}\right)\,.
\end{multline*}
If \(\gamma=d\) then \((S_{n-1})^{\tfrac{1}{\gamma-1}-\tfrac{1}{d-1}}=1\).
If on the other hand \(\gamma>d\) then the exponent \(\tfrac{1}{\gamma-1}-\tfrac{1}{d-1}\) is negative, while from the recurrence it follows that \(S_n\to0\) almost surely.
In either case, the above sum diverges if the same can be said for
\begin{equation}\label{eqn:infinite-sum}
 \sum_n \left((S_{n-1}P_{n-1}+T_n)^{\frac{1}{d-1}}-(S_{n-1}P_{n-1})^{\frac{1}{d-1}}\right)\,.
\end{equation}
However geometric ergodicity of \(X_n=S_nP_n\), together with independence of the \(T_i\)'s and \(U_i\)'s, and the exponential distribution of \(T_n\), allows us to deduce that 
almost surely infinitely many of these summands must satisfy
\[
 (S_{n-1}P_{n-1}+T_n)^{\frac{1}{d-1}}-(S_{n-1}P_{n-1})^{\frac{1}{d-1}} \quad>\quad 1\,.
\]
Hence it follows that the infinite sum \eqref{eqn:infinite-sum} almost surely diverges, and therefore \(y\) 
will almost surely takes infinite time to reach infinity. 
This suffices to prove the
theorem.
\end{proof}

\begin{rem}\label{rem:equivalence}
 It can be shown that in case \(\gamma<d\) then there are indeed \(\Pi\)-paths which reach infinity in finite time, and even in finite expected time.
\end{rem}

\begin{rem}
 We will see (Theorem \ref{thm:connection}) that if we require \(\gamma>d\) then in addition we can connect specified pairs of points by \(\Pi\)-paths in finite time. 
\end{rem}

We can now deduce the existence of an \emph{a priori} bound on the global Lipschitz constant of \(\Pi\)-paths of finite length begun in a fixed compact set. 
For any constant \(C\), the lines from \(\Pi\) meeting \(\ball(\origin,C)\) 
have speed-limits bounded above by \(\overline{V}(C)\), a random value depending on the random environment \(\Pi\). 
Fixing \(T<\infty\) and \(r_0>0\), and taking \(C\) to be the random value depending on \(\Pi\) in the statement of Theorem \ref{thm:a-priori-bound}, then
Remark \ref{rem:pi-path} Condition \ref{def:pi-path-condition3} implies that 
the \(\Pi\)-paths in \(\Paths_T\) beginning in \(\ball(\origin,r_0)\) satisfy a uniform Lipschitz property with random Lipschitz constant depending on \(\Pi\), \(r_0\), and \(T\).
Note that \(r_0>0\) can be chosen arbitrarily.

A closure result for \(\Paths_T\) follows immediately.
Recall (again using results expounded in \citealp[ch.5]{Evans-1998}) 
that the Sobolev space \(\Sobolev([0,T)\to\Reals^d)\) can be viewed as the space of absolutely continuous curves \(\xi:[0,T)\to\Reals^d\)
whose first derivatives \(\xi'\) satisfy \(\int_0^T|\xi'|^2\,\d{t}<\infty\). 
Thus \(\Sobolev([0,T)\to\Reals^d)\) forms a separable Hilbert space when furnished with an inner-product norm given by
\(\sqrt{\int_0^T|\xi|^2\,\d{t}+\int_0^T|\xi'|^2\,\d{t}}\). 
Recall also the Hilbert-space fact that bounded and weakly-closed subsets of \(\Sobolev([0,T)\to\Reals^d)\) are weakly compact.

\begin{lem}[Closure of path space]\label{lem:closure}
 Suppose \(\gamma\geq d\). For finite \(T\), the path space \(\Paths_T\) is closed in the Sobolev space \(\Sobolev([0,T)\to\Reals^d)\).
\end{lem}
\begin{proof}
 Consider any sequence \(\xi_1\), \(\xi_2\), \ldots of paths drawn from \(\Paths_T\). 
Suppose \(\xi_n\to \xi\) when considered as elements of \( \Sobolev([0,T)\to\Reals^d)\). By Sobolev space arguments,
taking a convergent subsequence if necessary, we may suppose that
\begin{align*}
 \xi_n(t) \quad&\to\quad \xi(t) \qquad \text{for all }t\in[0,T)\,;\\
 \xi_n'(t)\quad&\to\quad \xi'(t)\qquad \text{for almost all }t\in[0,T)\,.
\end{align*}
Consider the following time-set, which can be seen to be Lebesgue-null by 
combining the Lipschitz property of \(\xi\), the choice of the convergent sequence, and 
using the equivalent definition of \(\Pi\)-paths given by Remark \ref{rem:pi-path} Condition \ref{def:pi-path-condition2}:
\begin{align*}
 \mathcal{N}\quad=\quad\Big\{t\;:\; &\text{ either } \xi'(t) \text{ does not exist, or } \xi'_n(t)\not\to \xi'(t),\\
 &\text{ or, for some \(n\), } \xi'_n(t)\neq0 \text{ but }\xi_n(t)+\xi'_n(t)\Reals\not\in \Silhouette_{|\xi'_n(t)|}\Big\}\,.
\end{align*}

Fix attention on \(t\not\in\mathcal{N}\) for which \(|\xi'(t)|>u>0\). Then \(\xi_n'(t)\to\xi'(t)\), so 
\[
 \xi_n(t)+\xi'_n(t)\Reals \;\in\; \Pi_u \qquad\text{ for all large enough }n\,.
\]
Furthermore, 
since \(\Pi_u\) yields a locally finite line process,
\begin{align*}
\text{either }
& \xi_n(t)+\xi'_n(t)\Reals \text{ eventually equals } \ell_t \text{ for some fixed } \ell_t\text{ from }\Pi_u \\
& \qquad\qquad\qquad\text{ and moreover } \xi(t)+\xi'(t)\Reals=\ell_t \\
& \text{or } 
\xi_n(t) \quad\to\quad \text{ an intersection point of the line process }\Pi_u\,.
\end{align*}
In the first case, since \(\ell_t\) is a line from \(\Pi_{|\xi'_n(t)|}\) for all large enough \(n\),
it follows from the convergence \(\xi_n'(t)\to\xi'(t)\) that \(\ell_t\) is a line from \(\Pi_{|\xi'(t)|}\).

In the second case, Corollary \ref{cor:network-intersections} implies that
\[ 
 \left\{t\;:\; \xi(t) \text{ is an intersection point of \(\Pi\) and } \xi'(t)\neq0\right\}
\]
must be a null time-set. 
Either way,
up to a Lebesgue-null time-set,
\begin{itemize}
 \item[] either \(\xi'(t)=0\),
 \item[] or \(\xi(t)+\xi'(t)\Reals\in \Pi_{|\xi'(t)|}\).
\end{itemize}
Hence \(\xi\) is a \(\Pi\)-path (using Remark \ref{rem:pi-path}, specifically the equivalent defining Condition \ref{def:pi-path-condition2} for \(\Pi\)-paths), hence belongs to \(\Paths_T\).
\end{proof}

We can improve on this lemma to show that \(\Paths_t\) is \emph{weakly closed}, from which there follows a useful compactness result.
\begin{thm}[Weak closure of path space]\label{thm:weak-closure}
Suppose \(\gamma\geq d\). For finite \(T\),  the path space \(\Paths_T\) is weakly closed in \(\Sobolev([0,T)\to\Reals^d)\).
\end{thm}
\begin{proof}
  Consider any sequence \(\xi_1\), \(\xi_2\), \ldots in \(\Paths_T\). Suppose \(\xi_n\to \xi\) \emph{weakly} in \(\Sobolev([0,T)\to\Reals^d)\).
We will invoke the equivalent defining Condition 
\ref{def:pi-path-condition3} from Remark \ref{rem:pi-path}, namely, if we can show that \(|\xi'(t)|\leq V(\xi(t))\) for almost all \(t\in[0,T)\) then \(\xi\) is a \(\Pi\)-path.
As an immediate consequence, this will imply that \(\Paths_T\) is weakly closed.

If \(T<\infty\) then weak convergence in \(\Sobolev([0,T)\to\Reals^d)\) implies uniform convergence;
in particular this implies \(\xi\) is continuous.

Fix \(v>0\), and consider a non-empty time interval \([r,s]\subseteq[0,T)\), such that \([r,s]\) is contained in a single connected component of the open time-set 
\(\{t\in[0,T):\xi(t)\not\in\Silhouette_{v}\}\). Thus the compact set \(\{\xi(t):r\leq t\leq s\}\) does not intersect the closed set \(\Silhouette_{v}\)
and indeed
will not intersect the closed set \(\Silhouette_{v-\eps}\) for some (perhaps small) \(\eps\in(0,v)\).
By the uniform convergence of \(\xi_n\) to \(\xi\), 
we know that for all large enough \(n\) 
the compact set \(\{\xi_n(t):r\leq t\leq s\}\) does not intersect \(\Silhouette_{v-\eps}\); since \(\xi_n\) is a \(\Pi\)-path, this implies
that \(|\xi_n'(t)|<v-\eps\) for almost all \(t\in[r,s]\), and hence that \(\xi_n\) is Lipschitz with Lipschitz constant \(v-\eps\) over the time interval \([r,s]\).
But the uniform convergence of \(\xi_n\) to \(\xi\) implies that \(\xi\) itself is Lipschitz with Lipschitz constant \(v-\eps\) over the time interval \([r,s]\),
therefore that \(|\xi'(t)|\leq v-\eps<v\) for almost all \(t\in[r,s]\).

Expressing the open time-set 
\(\{t\in[0,T):\xi(t)\not\in\Silhouette_{v-\eps}\}\) as a countable union of closed intervals, 
and letting \(\eps\to0\),
it follows that for almost all \(t\in[0,T)\) if \(\xi(t)\not\in\Silhouette_{v}\)
then \(|\xi'(t)|<v\). Consequently \(|\xi'(t)|\leq V(\xi(t))\) for almost all \(t\in[0,T)\) as required.
\end{proof}

It follows immediately that bounded subsets of \(\Paths_T\) are weakly precompact, and this is crucial for the discussion of minimum-time \(\Pi\)-paths given in the
next section.
\begin{cor}[Weak compactness and path space]\label{cor:compactness}
Suppose \(\gamma\geq d\). Fix \(T>0\) and consider the set of paths in \(\Paths_T\) which begin in a compact set \(K\).
This set is weakly compact as a subset of \(\Sobolev([0,T)\to\Reals^d)\).
\end{cor}
\begin{proof}
 Immediate from Theorems \ref{thm:a-priori-bound} and \ref{thm:weak-closure} together with weak compactness results for
 the Hilbert space \(\Sobolev([0,T)\to\Reals^d)\).
\end{proof}

\section[\texorpdfstring{$\Pi$}{Π}-paths between point pairs]{\(\Pi\)-paths between arbitrary points}\label{sec:pi-paths and points}
We have shown that the condition \(\gamma\geq d\) ensures that almost surely \(\Pi\)-paths
do not diverge to infinity in finite time (as noted in Remark \ref{rem:equivalence}, this condition is also necessary). 
However it is not yet clear whether \(\Pi\)-paths can move from one point to another in finite time.
In dimension \(d=2\) the question is essentially one of whether one can reach a specified point in finite time by travelling along paths of progressively slower and slower speed.
Dimension \(d\geq3\) appears intransigent at first glance, since one knows that lines of a Poisson line process do not intersect each other in dimension \(3\) and higher.
Nevertheless, almost surely there are \(\Pi\)-paths connecting \emph{all} pairs of points in dimension \(2\) \emph{and} higher, so long as we strengthen the condition on \(\gamma\) to \(\gamma>d\).

We begin by showing how to connect specified pairs of points.
\begin{thm}[\(\Pi\)-paths connect given pairs of points in finite time]\label{thm:connection}
 Suppose \(\gamma>d\). 
Then specified \(x_1\) and \(x_2\) in \(\Reals^d\)
can almost surely be connected in finite time \(T\) by a \(\Pi\)-path \(\xi\).
\end{thm}
\begin{proof}
 The construction connects segments of lines from \(\Pi\) together in a tree-like fashion, rather than sequentially.
The basic idea is as follows:
for a sufficiently large constant \(\alpha\) (indeed, \(\alpha>2^{(\gamma-1)/(\gamma-d)}\) suffices), 
construct disjoint balls of radius \(|x_1-x_2|/\alpha\)  around \(x_1\) and \(x_2\). 
Choose the fastest line \(\ell\) in \(\Pi\)
hitting both balls, corresponding to the root of a binary tree representation of a path connecting \(x_1\) to \(x_2\). Then create two daughter nodes, repeating the construction
based on (a) \(x_1\) and the closest point to \(x_1\) on \(\ell\), and (b) \(x_2\) and the closest point to \(x_2\) on \(\ell\).
Extend this recursively to generate a binary tree-indexed collection of line segments. 
Figure \ref{fig:binary-tree} illustrates the first two stages of the \(d=2\) case.
\begin{Figure}
  \includegraphics[width=5in]{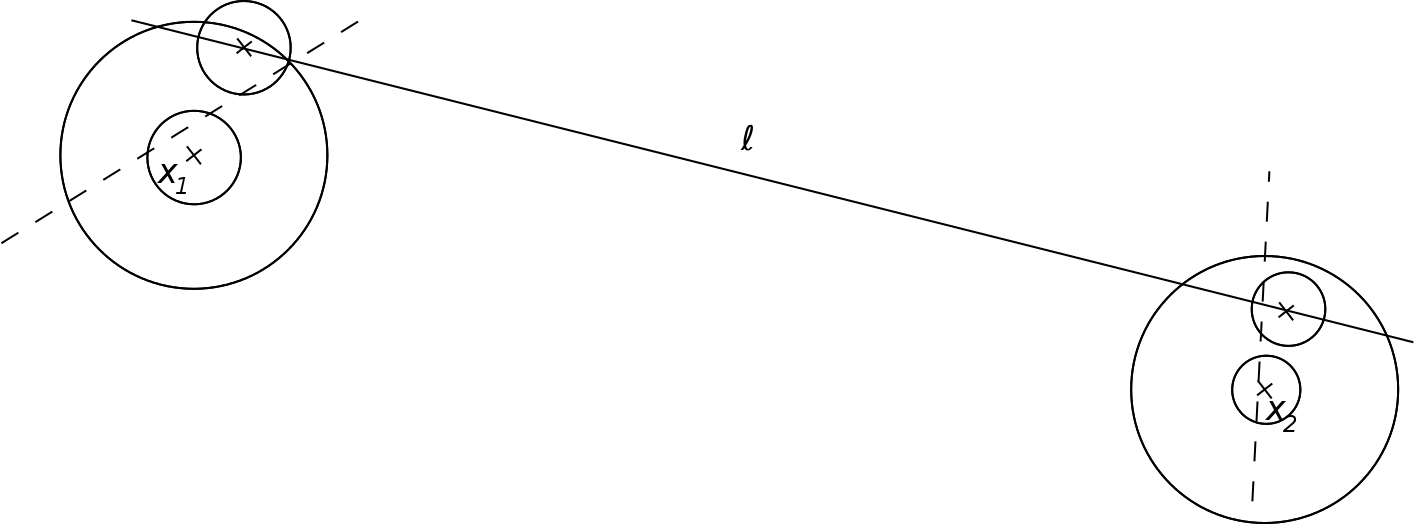}
\centering
\caption{\label{fig:binary-tree}
First two stages of a recursive construction of a \(\Pi\)-path from \(x_1\) to \(x_2\) in two dimensions,
using fastest available lines taken from an improper Poisson line process marked by speeds.
Note that in the case of higher dimensions the lines will almost surely not intersect.
}
\end{Figure}

The path \(\xi\) formed from this binary tree is evidently a \(\Pi\)-path. 
The issue is to show that it makes the connection from \(x_1\) to \(x_2\) in finite time.

Firstly, we need a stochastic lower bound for the speed-limit of the fastest line connecting two balls, namely the speed of the fastest line of \(\Pi\) in
\[
 \hittingset{\ball(x_1,\alpha^{-1}r)}\cap\hittingset{\ball(x_2,\alpha^{-1}r)}\,.
\]
Here we write \(r=|x_1-x_2|\) for the Euclidean distance between \(x_1\) and \(x_2\).
We obtain a stochastic lower bound for the speed distribution in two steps: 
\begin{itemize}
 \item[(a)] Shrink \(\ball(x_1,\alpha^{-1}r)\) to \(\ball(x_1,\alpha^{-1}r/2)\) and then replace \(\ball(x_1,\alpha^{-1}r/2)\)
by the hyper-disk \(D(x_1,\alpha^{-1}r/2)\) obtained by intersecting \(\ball(x_1,\alpha^{-1}r/2)\) with the hyperplane through \(x_1\) which is normal to the vector \(x_2-x_1\);
\item[(b)] Consider the bundle of lines in \(\hittingset{D(x_1,\alpha^{-1}r/2)}\cap\hittingset{\ball(x_2,\alpha^{-1}r)}\) which run through a given point 
\(z\in D(x_1,\alpha^{-1}r/2)\). For each such \(z\), reduce the bundle size by restricting attention to lines which additionally intersect 
\(\ball(z+x_2-x_1,\alpha^{-1}r/2)\subset \ball(x_2,\alpha^{-1}r)\).
\end{itemize}
This geometric construction is illustrated in Figure \ref{fig:reduction}.
\begin{Figure}
  \includegraphics[width=3.5in]{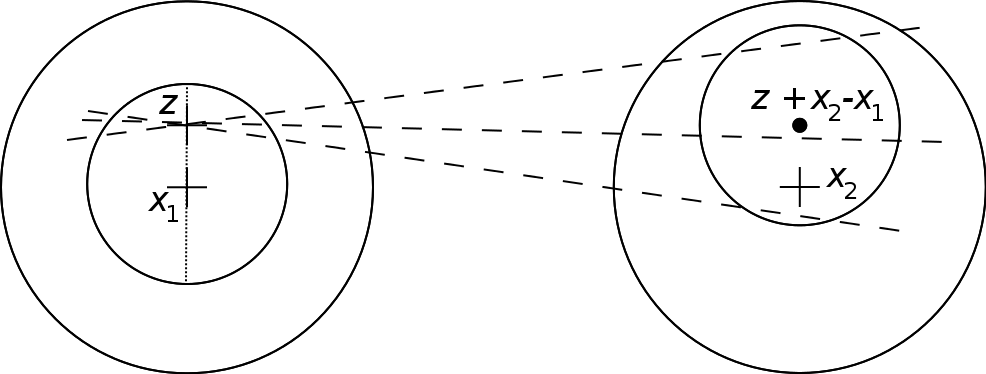}
\centering
\caption{\label{fig:reduction}
Reduction of \(\hittingset{\ball(x_1,\alpha^{-1}r)}\cap\hittingset{\ball(x_2,\alpha^{-1}r)}\) to a smaller hitting set for
which the line measure is more easily computable yet still provides a useful lower bound.
}
\end{Figure}

Using inclusion of hitting sets, this produces an easily computable lower bound for the line measure:
\begin{multline*}
 \mu_d\left(\hittingset{\ball(x_1,\alpha^{-1}r)}\cap\hittingset{\ball(x_2,\alpha^{-1}r)}\right)
\quad\geq\quad 
 \mu_d\left(\hittingset{D(x_1,\alpha^{-1}r/2)}\cap\hittingset{\ball(x_2,\alpha^{-1}r)}\right)\\
\quad\geq\quad
m_{d-1}\left(D(x_1,\alpha^{-1}r/2)\right)\times 
\mu^{(\origin)}_{d-1} \left(\hittingset{\ball(x_2-x_1,\alpha^{-1}r/2)}\right)\,.
\end{multline*}
Here \(\mu^{(\origin)}_{d-1}=\tfrac{\sin\theta}{\kappa_{d-1}}m_{S^{d-1}_+}\) is derived from the disintegration of \(\mu_d\) by Lebesgue measure on the hyperplane through \(x_1\) which is normal to \(x_2-x_1\),
using \eqref{eqn:disintegration2}. Thus \(\mu^{(\origin)}_{d-1}\) is a weighted version of
the invariant (hyper-surface area) measure on the hemisphere of un-sensed lines \(\ell\) passing through the origin \(\origin\), weighted by \(\sin\theta\) where
\(\theta\) is the angle between \(\ell\) and the hyperplane, normalized to have unit total measure.

Recall that \(\kappa_{d-1}\) denotes the \((\d-1)\)-volume of the unit \((d-2)\)-ball. Thus
\[
 m_{d-1}\left(D(\xi_1,\alpha^{-1}r/2)\right) \quad=\quad \kappa_{d-1} \left(\frac{r}{2\alpha}\right)^{d-1}\,,
\]
On the other hand, from \eqref{eqn:disintegration2}
the relevant computation of weighted hyper-surface area for the visibility hemisphere is
\[
\frac{\omega_{d-2}}{\kappa_{d-1}} \int^{\pi_2}_{\pi/2-\theta_0} \sin\theta  \,\cos^{d-2}\theta\,\d{\theta} \quad=\quad \frac{\omega_{d-2}}{(d-1)\kappa_{d-1}} \,\cos^{d-1}\theta_0
\quad=\quad \cos^{d-1}\theta_0\,,
\]
where \(\cos\theta_0=1/(2\alpha)\) (and noting that \(\omega_{d-2}=(d-1)\kappa_{d-1}\)); hence
\[
 \mu^{(\origin)}_{d-1} \left(\hittingset{\ball(x_2-x_1,\alpha^{-1}r/2)}\right)\quad=\quad \left(\frac{1}{2\alpha}\right)^{d-1}\,.
\]
These considerations yield the lower bound
\[
 \mu_d\left(\hittingset{\ball(\xi_1,\alpha^{-1}r)}\cap\hittingset{\ball(\xi_2,\alpha^{-1}r)}\right)
\quad\geq\quad 
\kappa_{d-1}\left(\frac{r}{4\alpha^2}\right)^{d-1}\,.
\]

This reasoning can be applied to the recursive construction indicated above. Let \(r_h\) be the distance between points at node \(h\) on the binary tree representing the path, then
(omitting some implicit conditioning on \(r_h\))
\[
 \Prob{\text{ fastest line speed at } h \leq v_h}
\quad\leq\quad\exp\left(
-\frac{\kappa_{d-1}}{(4\alpha^2)^{d-1}}\,\frac{r_h^{d-1}}{v_h^{\gamma-1}}
\right)\,.
\]
By construction, if node \(h\) is at level \(n\) of the tree then \(r_n\leq \alpha^{-n}r_0\), where \(r_0\) is the Euclidean distance between the original
start and finish points. Fixing \(\varepsilon>0\), we set
\[
 v_h \quad=\quad \frac{r_h^{(d-1)/(\gamma-1)}}{(n \zeta)^{\frac{1}{\gamma-1}}} \qquad \text{where }\zeta=\frac{(4\alpha^2)^{d-1}}{\kappa_{d-1}}(1+\varepsilon)\log 2
\]
Then \(r_h^{d-1}/v_h^{\gamma-1}=n \zeta\). Use the first Borel-Cantelli lemma, and the convergence of
\begin{multline*}
 \sum_h \exp\left(
-\frac{\kappa_{d-1}}{(d-1)(4\alpha^2)^{d-1}}\,\frac{r_h^{d-1}}{v_h^{\gamma-1}}
\right)
\;=\; \sum_n 2^n \exp\left(-(1+\varepsilon) n \log 2\right)
\;=\; \sum_n 2^{-\varepsilon n}\,<\, \infty\,,
\end{multline*}
to deduce that it is
almost surely the case that the speed limits of all but finitely many segments \(h\) in the binary tree representation
will exceed
\[
 v_h \quad=\quad \frac{r_h^{(d-1)/(\gamma-1)}}{(n \zeta)^{\frac{1}{\gamma-1}}}\,.
\]
By the triangle inequality the relevant path length for node \(h\) is no greater than \((1+\tfrac{2}{\alpha})r_h\). 
So the total time spent traversing
the path is finite when
\[
 \sum_h \left(1+\frac{2}{\alpha}\right) \frac{r_h}{v_h}
\quad=\quad
\sum_n 2^n \times \left(1+\frac{2}{\alpha}\right) (n \zeta)^{\frac{1}{\gamma-1}} \left(\frac{r_0}{\alpha^n}\right)^{(\gamma-d)/(\gamma-1)}
\quad<\quad\infty\,.
\]
But this sum converges when \(\alpha> 2^{(\gamma-1)/(\gamma-d)}\):
thus in this case the construction gives a finite-time \(\Pi\)-path 
between \(x_1\) and \(x_2\).
\end{proof}
\begin{rem}\label{rem:connection-by-geodesics}
It can be shown that \(\gamma>d\) is also a necessary condition for connection of \(x_1\) to \(x_2\) by a \(\Pi\)-path in \(\Reals^d\). 
For if \(\gamma=d\) then all \(\Pi\)-paths leaving the origin are subject to an upper bound using the comparison process of Theorem \ref{thm:a-priori-bound},
and a direct calculation shows that this comparison process takes infinite expected time to leave the origin.
\end{rem}

Note that, for \(d>2\), the finite-time path has a curious fractal-like property: whenever the \(\Pi\)-path changes from one line of positive speed-limit to another, 
then it must shift gears right down to zero speed then right up again to the new speed. (And the same applies to each change of speed-limit whilst shifting gears down, and so on \emph{ad infinitum}, as in the case of the fleas and poets of \citealp[On Poetry: a Rhapsody]{Swift-1733}.)

Since there are only countably many lines in \(\Pi\), a simple modification of the above construction shows that almost surely all lines in \(\Pi\) are connected.

\begin{cor}\label{cor:connect-all-lines}
If \(\gamma>d\) then almost surely all lines of \(\Pi\) are joined by finite-time \(\Pi\)-paths.
\end{cor}

If two points \(x_1\) and \(x_2\) are joined by \(\Pi\)-paths taking finite time, then it is reasonable to ask whether there is a minimum-time \(\Pi\)-path. As a consequence of 
Corollary \ref{cor:compactness}, we know that this occurs, since \(\gamma>d\) and hence \emph{a fortiori} the compactness condition \(\gamma\geq d\) holds. 
We summarize this conclusion by means of a definition and a further corollary.
We note in passing that \(\Pi\)-geodesics inherit all the properties of minimal geodesics in metric spaces; for example a minimal \(\Pi\)-geodesic cannot intersect itself.

\begin{defn}[\(\Pi\)-geodesic]\label{def:geodesic}
 The \(\Pi\)-path \(\xi:[0,T]\to\Reals^d\) is said to be a \emph{\(\Pi\)-geodesic} (or
\emph{minimum-time geodesic}) from \(\xi(0)=x_1\) to \(\xi(T)=x_2\) if there are no \(\Pi\)-paths connecting \(x_1\)
and \(x_2\) in \(\Paths_S\) for \(S<T\).
\end{defn}

\begin{cor}[Existence of \(\Pi\)-geodesics]\label{cor:geodesics}
Suppose \(\gamma>d\). Consider the set of paths \(\xi\) in \(\Paths\) which begin at fixed location \(\xi(0)=x_1\) and end at fixed location \(\xi(T)=x_2\)
(here \(T\) depends on \(\xi\)).
Almost surely there exist \(\Pi\)-geodesics in \(\Paths\) from \(x_1\) to \(x_2\).
\end{cor}
\begin{proof}
 By Theorem \ref{thm:connection}, almost surely there are connecting \(\Pi\)-paths in \(\Paths_T\) for large enough \(T<\infty\). Consider a sequence of such paths
\(\xi_1\), \(\xi_2\), \ldots,
starting at \(x_1\) and ending at \(x_2\), such that \(\xi_n(t)=x_2\) for all \(t\in[T_n,T)\), and such that \(T_n\) tends to \(T_\infty\) the infimum of all connection times for \(\Pi\)-paths
between \(x_1\) and \(x_2\). By Corollary \ref{cor:compactness} we may extract a weakly convergent subsequence, and the limit \(\xi\in\Paths_T\) 
will satisfy \(\xi(t)=x_2\) for all \(t\in[T_\infty,T)\) and hence realize the infimum. The
resulting \(\Pi\)-path \(\xi\) will be a \(\Pi\)-geodesic between \(x_1\) and \(x_2\).
\end{proof}

In the next section we will examine the extent to which \(\Pi\)-geodesics are uniquely determined by their end-points. 
Before turning to this matter,
we improve on Theorem \ref{thm:connection} by showing that if \(\gamma>d\) then almost surely \emph{all} pairs of points in \(\Reals^d\)
are connected by finite-time \(\Pi\)-paths. 
That is to say, almost surely 
there are no infinite singularities in the metric space induced by the time taken by fastest \(\Pi\)-path transit.
The proof closely follows that of Theorem \ref{thm:connection}, but splits paths apart in a hierarchical way so as to access entire regions rather than single points.

\begin{thm}\label{thm:metric-space}
Suppose \(\gamma>d\) and \(d\geq2\). With probability \(1\), the network formed by \(\Pi\) connects up all pairs of points in \(\Reals^d\) using
finite-time \(\Pi\)-paths.
 \end{thm}
\begin{proof}
It suffices to establish that, almost surely, finite-time \(\Pi\)-paths can be used to connect a specified point to all the points of a single hypercube of positive area.

Consider then the construction of a path \(\xi\) from \(x_1\) to a hypercube centred on \(x_2\), where \(|x_1-x_2|=r_0\) and the hypercube is of
side-length \(\alpha^{-1}r_0\) for some sufficiently large integer \(\alpha\) (indeed, \(\alpha>(1+\tfrac{\sqrt{d}}{2})2^{(\gamma-1)/(\gamma-d)}\) suffices). 
The construction commences as in Theorem \ref{thm:connection},
choosing the fastest line \(\ell\) of \(\Pi\) in
\(\hittingset{\ball(x_1,\alpha^{-1}r_0)}\cap\hittingset{\ball(x_2,\alpha^{-1}r_0)}\), and this corresponds to the root of a tree now representing a whole family of paths.
Repeat the construction, adding a further line from \(\Pi\) which nearly connects \(x_1\) to the point on \(\ell\) closest to \(x_1\), as in Theorem \ref{thm:connection}.
However on the other side we generate \(\alpha^d\) separate paths, using the fastest possible line to connect 
the ball of radius \(\alpha^{-1}r_0\) centred on the point on \(\ell\) closest to \(x_2\), to each of a total of \(\alpha^d\)
balls of radius \(\alpha^{-1}r_0\) centred on
centroids of cells arising from
a dissection of the original hypercube of side-length \(\alpha^{-1}r_0\) into \(\alpha^d\) sub-hypercubes each of side-length \(\alpha^{-2}r_0\).
This 
is illustrated in Figure \ref{fig:multiple}.
\begin{Figure}
  \includegraphics[width=3.5in]{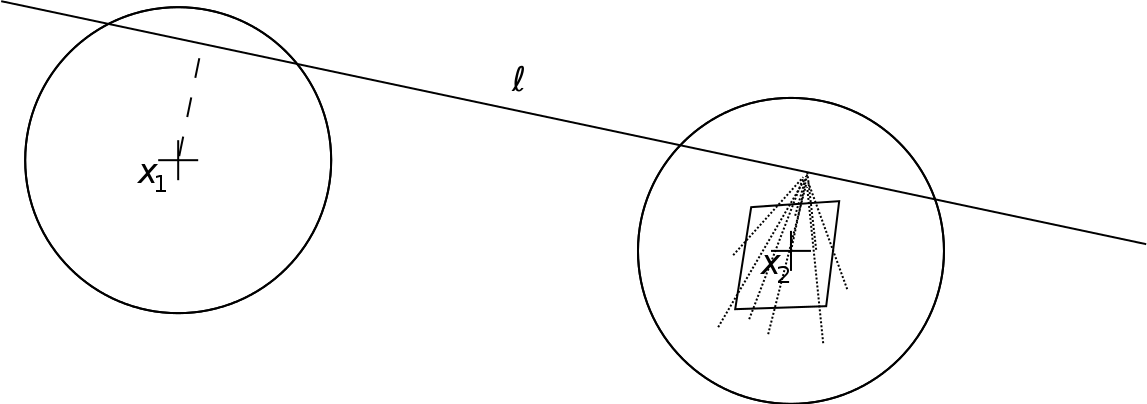}
\centering
\caption{\label{fig:multiple}
Initial stage of connecting \(x_1\) to points in a specified hypercube (case of \(d=2\)).
}
\end{Figure}

In the first case the new distance is at most \(\alpha^{-1} r_0\). In the second case we may use Pythagoras to show that the new distance is at most \((1+\half\sqrt{d})\alpha^{-1}r_0\).
This construction step generates \(1+\alpha^d\) new segments at the second level of the tree, the first one being like the segments generated in Theorem \ref{thm:connection},
while the remaining \(\alpha^d\) nearly connect a given point to \(\alpha^d\) centroids of sub-hypercubes.

Repeating the construction down to level \(n\), we generate \(h_n\) segments of the first kind and \(k_n\) segments of the second kind, where
\begin{align}\label{eqn:binary-recursion}
 h_n \quad&=\quad 2 h_{n-1} + k_{n-1}\,, & \qquad & h_1 \quad=\quad 1\,,\\
 k_n \quad&=\quad \alpha^d k_{n-1}\,,    & \qquad & k_1 \quad=\quad \alpha^d\,.\nonumber
\end{align}
The total number of segments which have been built at level \(n\) is therefore
\[
 h_n+k_n \quad=\quad  2 (h_{n-1} + k_{n-1}) + (\alpha^d - 1)k_{n-1}\,.
\]
The bound on \(\alpha\) imposed at the start of the proof, together with \(\gamma>d\geq2\), shows that \(\alpha^d>8\), 
and so we can use the recursion \eqref{eqn:binary-recursion} and the fact that \(h_1+k_1=1+\alpha^d\) to deduce
\begin{multline}
  h_n+k_n \quad\leq\quad 2 (h_{n-1} + k_{n-1}) + (\alpha^d - 1)k_{n-1}\\
\quad=\quad 2^2 (h_{n-2} + k_{n-2}) +  (\alpha^d - 1) (k_{n-1} + 2 k_{n-2})\\
\quad=\ldots=\quad 2^{n-1} (h_{1} + k_{1}) +  (\alpha^d - 1) (k_{n-1} + 2 k_{n-2} + \ldots + 2^{n-2} k_{1})
\quad\leq\quad
\text{constant}\times \alpha^{nd}\,.
\end{multline}

Now consider the speed-limit of the line forming node \(h\) at level \(n\). Suppose the relevant distance between target points is \(r_h\). Then (as in Theorem \ref{thm:connection})
\[
 \Prob{\text{ fastest line speed-limit at } h \leq v_h}
\quad\leq\quad\exp\left(
-\frac{\kappa_{d-1}}{(4\alpha^2)^{d-1}}\,\frac{r_h^{d-1}}{v_h^{\gamma-1}}
\right)\
\]
We know
\[
 r_h \quad\leq\quad \left(\frac{1+\sqrt{d}/2}{\alpha}\right)^n r_0\,.
\]
Choose 
\[
 v_h \quad=\quad \frac{r_h^{(d-1)/(\gamma-1)}}{(n \zeta)^{\frac{1}{\gamma-1}}} \qquad \text{as before, but with }\zeta=\frac{(4\alpha^2)^{d-1}}{\kappa_{d-1}}(1+\varepsilon)d\log \alpha
\]
Using the first Borel-Cantelli lemma once again, all but finitely many of the segments in this construction have speed-limit exceeding the respective \(v_h\), since
\[
 \sum_n \alpha^{dn} \exp\left(-(1+\varepsilon)nd\log \alpha)\right) \quad=\quad \sum_n \alpha^{-\varepsilon n d}\quad<\quad\infty\,.
\]
Thus each one of the paths can be traversed in finite time if the following sum converges:
\begin{multline*}
 \sum_{h \text{ in specified path}} \frac{r_h}{v_h} \quad\leq\quad \sum_n 2^n (n \zeta )^{\frac{1}{\gamma-1}} \left(\frac{1+\sqrt{d}/2}{\alpha}\right)^{n(\gamma-d)/(\gamma-1)}\\
\quad\leq\quad \zeta^{\frac{1}{\gamma-1}}\sum_n n^{\frac{1}{\gamma-1}} \left(2  \left(\frac{1+\sqrt{d}/2}{\alpha}\right)^{(\gamma-d)/(\gamma-1)}\right)^n\,.
\end{multline*}
Recalling the stipulation that \(\gamma>d\), this converges if we choose
\[
 \alpha \quad>\quad \left(1+\frac{\sqrt{d}}{2}\right) 2^{(\gamma-1)(\gamma-d)}\,.
\]
The family of \(\Pi\)-paths used here is weakly compact (Corollary \ref{cor:compactness}).
It follows therefore that this construction almost surely delivers \(\Pi\)-paths which within finite time connect a specified point \(x_1\) to all points in a non-empty hypercube with centroid \(x_2\) and side-length \(\alpha^{-1}|x_2-x_1|\).
Using this fact together with judicious
concatenation of \(\Pi\)-paths, it follows
that almost surely all pairs of points in \(\Reals^d\) are connected by \(\Pi\)-geodesics.
\end{proof}

In the case \(d=2\), both Theorems \ref{thm:connection} and \ref{thm:metric-space} can be proved more directly, 
exploiting the fact that non-parallel lines in \(\Reals^2\) always meet. 
The resulting \(\Pi\)-paths are then formed from consecutive sequences of line segments drawn from \(\Pi\).
However our interest is in \emph{\(\Pi\)-geodesics}, and even in case \(d=2\) it is not yet known whether \(\Pi\)-geodesics can be constructed as 
consecutive sequences of line segments.

While we define \(\Pi\)-geodesics as minimum-\emph{time} paths, we retain an interest in the actual lengths of \(\Pi\)-geodesics.
It is a consequence of Theorem \ref{thm:a-priori-bound} that a \(\Pi\)-geodesic between two points is almost surely of finite length:
this follows because if the \(\Pi\)-geodesic has finite duration \(T\)
then it must be contained in a sufficiently large ball, and therefore its maximum speed is bounded, which in turn bounds the length.
There is a more subtle question, namely whether the \emph{mean} length of the \(\Pi\)-geodesic is finite. 
We shall answer this question in the affirmative in Section \ref{sec:pi-geodesics-finite-mean-length}, but only for the case of dimension \(d=2\).
 
The above results establish the existence of \(\Pi\)-geodesics, but only non-constructively. The
principal difficulty in taking a
constructive approach lies in the implicit tree-like way in which \(\Pi\)-paths are constructed
in Theorems \ref{thm:connection} and \ref{thm:metric-space}. In the remainder of this section we
show how to approximate \(\Pi\)-paths by sequentially-defined Lipschitz paths which are almost \(\Pi\)-paths.
The major benefit of this result is that it implies the measurability of the random time which a \(\Pi\)-geodesic would take to move from one specified point \(x\) to another specified point \(y\). 
As is commonly the case for measurability arguments, the details are a little tedious; however the result does provide theoretical justification
for some simulation constructions of \(\Pi\)-geodesics (for example, the construction in Figure \ref{fig:dumbbell}).
The essence of the matter is to work
with Lipschitz paths which would be \(\Pi\)-paths if the upper-semicontinuous speed-limit \(V\) were
replaced by \(\max\{\eps,V\}\) for some small \(\eps>0\).
The methods of proof of the following results also justify the simulation algorithm used to produce the realizations of networks in Figure \ref{fig:dumbbell}.

\begin{defn}[\(\eps\)-near-sequential-\(\Pi\)-path]
\label{def:near-sequential-pi-path}
For given \(\eps>0\), a continuous path \(\widetilde{\xi}:[0,T]\to\Reals^d\) is an
\emph{\(\eps\)-near-sequential-\(\Pi\)-path} if there is a finite dissection of the interval \([0,T]\) as
\[
 0=b_0\leq a_1\leq b_1\leq a_2\leq b_2\leq \ldots\leq a_m\leq b_m\leq a_{m+1}=T\,,
\]
associated with a finite sequence of marked lines from \(\Pi\) (possibly with repeats),
\[
 (\widetilde{\ell}_1,\widetilde{v}_1),  (\widetilde{\ell}_2,\widetilde{v}_2), \ldots,  (\widetilde{\ell}_m,\widetilde{v}_m)\,,
\]
such that
\begin{itemize}
 \item[(a)] \(\widetilde{\xi}(t)\in\widetilde{\ell}_r\) when \(a_r\leq t\leq b_r\), and
\(|\widetilde{\xi}'(t)|\leq \widetilde{v}_r\) for almost all \(t\in(a_r,b_r)\);
\item[(b)] \(|\widetilde{\xi}'(t)|< \eps\) for almost all \(t\in\bigcup_{r=0}^m[b_r,a_{r+1}]\)
and \(\widetilde{\xi}'\) is constant on each \([b_r,a_{r+1}]\);
\item[(c)] \(\sum_{r=0}^m |a_{r+1}-b_r|< \eps\).
\end{itemize}
\end{defn}
 
The next result shows that \(\Pi\)-paths can be approximated by \(\eps\)-near-sequential-\(\Pi\)-paths for small \(\eps>0\), simply by
using the principal marked lines involved in \(\xi\) to generate the finite marked line sequence
\((\widetilde{\ell}_1,\widetilde{v}_1)\),  \((\widetilde{\ell}_2,\widetilde{v}_2)\), \ldots,  \((\widetilde{\ell}_m,\widetilde{v}_m)\).

\begin{thm}\label{thm:pi-path-approximation}
 Suppose only that \(\gamma> 1\), so that the line processes \(\Pi_v\) are locally finite for each \(v>0\).
Consider a \(\Pi\)-path \(\xi:[0,T]\to\Reals^d\), defined up to some fixed finite time \(T\) and running from \(x_1\) to \(x_2\).
For each \(\eps>0\) there can be found an \(\eps\)-near-sequential-\(\Pi\)-path \(\widetilde{\xi}:[0,T]\to\Reals^d\) such that
\begin{itemize}
 \item[] \(\widetilde{\xi}(0)=x_1\), \(\widetilde{\xi}(T)=x_2\);
\item[] \(\sup\{|\xi(t)-\widetilde{\xi}(t)|:t\in[0,T]\}<\eps\).
\end{itemize}
\end{thm}
Before proving this, we state the following important corollary, whose proof is an immediate consequence.
\begin{cor}\label{cor:pi-path-approximation}
 Suppose only that \(\gamma> 1\). Every \(\Pi\)-path defined up to finite time \(T\) can be uniformly 
approximated by a sequence of \(\eps\)-near-sequential-\(\Pi\)-paths such that \(\eps\downarrow0\) along the sequence.
\end{cor}
\begin{proof}[Proof of Theorem \ref{thm:pi-path-approximation}]
 It suffices to prove the result for a fixed positive \(\eps<1\).

The Poisson line process \(\Pi\) has no triple intersections, and therefore a
given \(\Pi\)-path \(\xi\) can only ever lie on at most two lines simultaneously. Hence by countable exhaustion (based on ordering by \(\Leb\{t\in[0,T]:\xi(t)\in\ell\}\)) it follows that there can be only countably many marked lines \((\ell,v)\in\Pi\) such that \(\Leb\{t\in[0,T]:\xi(t)\in\ell\}>0\) (so that \(\ell\) and \(\xi\) intersect in a time-set of positive measure).

Since \(\xi\) is continuous, we know that 
the image \(\Image(\xi)\) of \(\xi\) is compact, and therefore (since \(\Pi_v\) is locally finite for any positive \(v\)) the lines intersecting \(\xi\) in time-sets of positive measure can be sequentially ordered by speed in decreasing order:
\begin{align}
               & (\ell_1, v_1), \quad         (\ell_2, v_2), \quad         (\ell_3, v_3), \quad         \ldots \label{eq:positive-decreasing-order}\\
  \text{ and } \quad & v_1        \quad\geq\quad     v_2        \quad\geq\quad          v_3   \quad\geq\quad \ldots \nonumber \,.
\end{align}
Note that we do \emph{not} presume that the \(\ell_i\) are visited sequentially in order.

It is an immediate consequence of the above that \(\xi\) is \(\Lipschitz(v_1)\). Moreover a consequence of the countable exhaustion 
construction is that 
\[
 \Leb\{t\in[0,T]: \xi(t)\not\in \ell_1\cup\ell_2\cup\ldots\cup\ell_n\} \to 0 \text{ as }n\to\infty\,.
\]
Therefore for all \(\eps>0\), for all sufficiently large \(n\),
\begin{align*}
                   &\Leb\{t\in[0,T]: \xi(t)\not\in \ell_1\cup\ell_2\cup\ldots\cup\ell_n\} \quad<\quad\eps\,,
\\
 \text{ and } \quad&\qquad v_{n+1} \quad<\quad \eps\,. 
\end{align*}
We use the finite sequence \((\ell_1,v_1)\), \((\ell_2,v_2)\), \ldots \((\ell_n,v_n)\)
to generate an \(\eps\)-near-sequential-\(\Pi\)-path \(\widetilde{\xi}\) which approximates \(\xi\) in uniform norm.
Note that this finite sequence is \emph{not} the same as the finite sequence
\((\widetilde{\ell}_1,\widetilde{v}_1)\),  \((\widetilde{\ell}_2,\widetilde{v}_2)\), \ldots,  \((\widetilde{\ell}_m,\widetilde{v}_m)\)
from Definition \ref{def:near-sequential-pi-path}, but is used to construct it iteratively.

We begin by setting \(\widetilde{\xi}(0)=\xi(0)\) and \(\widetilde{\xi}(T)=\xi(T)\).

Consider first the time set \(\{t\in[0,T]:\xi(t)\in\ell_1\}\). The \(\Pi\)-path \(\xi\) makes a countable number of excursions
away from \(\ell\), and the excursion intervals form the connected components of the (relatively) open set
\([0,T]\setminus\{t\in[0,T]:\xi(t)\in\ell_1\}\). 
There are at most two incomplete excursions in the interval \([0,T]\), namely the beginning excursion,
anchored to \(x_1\) at time \(0\) on the left, and the end excursion, anchored to \(x_2\) at time \(T\) on the right. In addition there can be at most finitely 
many complete excursions for which \(\dist(\xi,\ell)\) reaches the level \(\eps\) (in fact a calculation using the \(\Lipschitz(v_1)\)
property of \(\xi\) gives an upper bound on the number of such excursions, namely \(\half v_1 T/\eps\)). We set
\begin{align*}
& U_0 \quad=\quad [0,T]\,,\\
& U_1 \quad=\quad 
\\
&\{t\in U_0:\xi(t)\not\in\ell_1\}
\setminus \bigcup\{(a,b)\subset U_0: \xi(a), \xi(b)\in\ell_1, \; 0<\dist(\xi(t),\ell_1)<\eps \text{ for }a<t<b\}\,.
\end{align*}
So \(U_1\) is a finite union of intervals (relatively open in \(U_0\)).

Define \(\widetilde{\xi}\) on \(U_0\setminus U_1\) as the orthogonal projection of \(\xi\) on \(\ell_1\). By the properties of 
orthogonal projection and the \(\Lipschitz(v_1)\) property of \(\xi\) on \(U_0\), it follows that \(|\widetilde{\xi}'(t)|\leq v_1\)
for almost all \(t\in U_ 0\setminus U_1\). Moreover, by construction, 
\begin{align*}
|\widetilde{\xi}(t)-\xi(t)| \quad&=\quad 0 & \text{ if } \xi(t)\in \ell_1 \text{ and } t\in U_0\setminus U_1 \,,\\
|\widetilde{\xi}(t)-\xi(t)| \quad&<\quad \eps & \text{ for other } t \in U_0\setminus U_1\,.
\end{align*}
Hence  \(|\widetilde{\xi}(t)-\xi(t)|<\eps\) for all \(t \in U_0\setminus U_1\). Finally, note that
\(\{t\in[0,T]:\xi(t)\in\ell_1\}\subseteq [0,T]\setminus U_1\).

Now define 
\begin{align*}
& U_2 \quad=\quad \\
& \{t\in U_1:\xi(t)\not\in\ell_2\}
\setminus \bigcup\{(a,b)\subset U_1: \xi(a), \xi(b)\in\ell_2, \; 0<\dist(\xi(t),\ell_2)<\eps \text{ for }a<t<b\}\,,
\end{align*}
and note that by construction \(\ell_1\) cannot intersect \(\xi\) in a time-set of positive measure in \(U_1\), so that \(\xi\) is
\(\Lipschitz(v_2)\) in the time-set \(U_1\). We can argue as before that \(U_2\) is a finite union of relatively open intervals.
We can extend the definition of \(\widetilde{\xi}\) to \(U_1\setminus U_2\) by using orthogonal projection of \(\xi\) onto \(\ell_2\):
we have  \(|\widetilde{\xi}'(t)|\leq v_2\)
for almost all \(t\in U_1\setminus U_2\), and \(|\widetilde{\xi}(t)-\xi(t)|<\eps\) for all \(t \in U_1\setminus U_2\).
Moreover, \(\{t\in[0,T]:\xi(t)\in\ell_1\cup\ell_2\}\subseteq [0,T]\setminus U_2\).

The construction is illustrated in Figure \ref{fig:near-pi}
\begin{Figure}
  \includegraphics[width=3in]{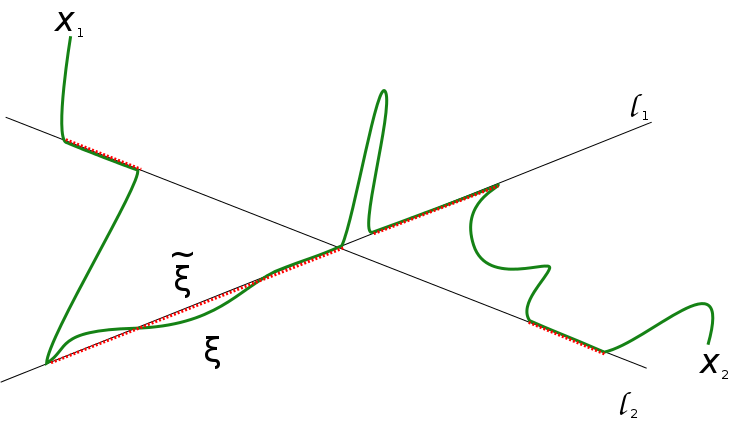}
\centering
\caption{\label{fig:near-pi}
First two stages of the iterative construction of an \(\eps\)-near-sequential-\(\Pi\)-path. The solid curve
represents the trajectory of the \(\Pi\)-path \(\xi\). The dotted segments represent the partially-defined trajectory of \(\widetilde{\xi}\)
after these first two stages (later stages successively fill in the gaps).
Note that \(\widetilde{\xi}\) is defined (a) when \(\xi\) runs along one of the lines \(\ell_1\), \(\ell_2\) and (b) when 
\(\xi\) makes small excursions from one of these lines.             
}
\end{Figure}

Iterating this construction, we end up defining \(\widetilde{\xi}\) on a time-set \([0,T]\setminus U_n\) 
containing \(\{t\in[0,T]:\xi(t)\in \ell_1\cup\ldots\cup\ell_n\}\), such that
\(\widetilde{\xi}(t)\in\ell_1\cup\ell_2\cup\ldots\cup\ell_n\) for \(t\in[0,T]\setminus U_n\),
with
\(|\widetilde{\xi}'(t)|\leq v_r\)
for almost all \(t\) such that \(\widetilde{\xi}(t)\in \ell_r\), for \(r=1\), \(2\), \ldots \(n\), and finally
\(|\widetilde{\xi}(t)-\xi(t)|<\eps\) if \(t\in[0,T]\setminus U_n\). Since
\(\{t\in[0,T]:\xi(t)\in \ell_1\cup\ldots\cup\ell_n\}\subseteq[0,T]\setminus U_n\), it follows from the countable exhaustion construction
that \(\Leb([0,T]\setminus U_n)<\eps\). 

We next complete the construction on the finite family of excursion intervals which are the connected components of the relatively
open set \(U_n\). 
Note that by construction \(\widetilde{\xi}\) agrees with \(\xi\) on the end-points of these excursion intervals.
None of the lines \(\ell_1\), \(\ell_2\), \ldots, \(\ell_n\) intersect \(U_n\) in a time-set of positive measure:
therefore \(\xi\) satisfies a \(\Lipschitz(v_{n+1})\) property on \(U_n\). Hence for each of these excursion intervals, if \(a\) and \(b\) are the 
end-points then \(|\xi(b)-\xi(a)|\leq (b-a)v_{n+1}<(b-a)\eps\). 

Accordingly we can define \(\widetilde{\xi}\)
by linear interpolation over the excursion interval
(so that \(\widetilde{\xi}'\) is indeed constant over this excursion interval), 
with the result that \(|\widetilde{\xi}'(t)|<\eps\) for almost all \(t\in[a,b]\). Finally
the \(\Lipschitz(v_{n+1})\) property implies that \(|\xi(t)-\xi(a)|\) and \(|\xi(t)-\xi(b)|\) are both strictly bounded above
by \(|b-a|\eps \leq \eps^2 \leq \eps\) when \(a<t>b\) (use \(|b-a|\leq \Leb([0,T]\setminus U_n)<\eps\) and \(\eps\leq1\)); it follows by convexity that the same bound holds 
if \(a\) and \(b\) are replaced by the piecewise interpolant \(\widetilde{\xi}(t)\):
\[
 |\xi(t)-\widetilde{\xi}(t)|\quad<\quad \eps\,.
\]
It follows that \(\widetilde{\xi}\) is the required \(\eps\)-near-sequential-\(\Pi\)-path approximating \(\xi\) to within \(\eps\)
in uniform norm. The sequence 
\((\widetilde{\ell}_1,\widetilde{v}_1)\),  \((\widetilde{\ell}_2,\widetilde{v}_2)\), \ldots,  \((\widetilde{\ell}_m,\widetilde{v}_m)\)
is obtained from the successive visits (with repetitions) of \(\widetilde{\xi}\) to the finite sequence of lines
\((\ell_1,v_1)\), \((\ell_2,v_2)\), \ldots \((\ell_n,v_n)\).
\end{proof}

\begin{rem}\label{rem:geodesic-simplification}
 If \(\xi\) is a \(\Pi\)-geodesic then the above construction can be simplified. Using the notation of the proof, suppose
that \((a,b)\) is a connected component of \(U_n\). The maximum speed of \(\xi\)
in \(U_n\) is \(v_{n+1}\): consequently if \(\xi(s)\) and \(\xi(t)\) belong to \(\ell_{n+1}\) for \(s<t\), both 
belonging to \((a,b)\),
then the fastest route from \(\xi(s)\) to \(\xi(t)\) must run along \(\ell_{n+1}\) at maximum permitted speed \(v_n\). Consequently \(\{t\in(a,b):\xi\in\ell_{n+1}\}\)
must already be a relatively closed interval in \((a,b)\), so that there is no need to use the excursion construction in the proof of the theorem.
\end{rem}

We have seen that, under the weak condition \(\gamma>1\), 
every \(\Pi\)-path can be uniformly approximated by a sequence of \(\eps\)-near-sequential-\(\Pi\)-paths with \(\eps\downarrow0\).
Conversely, if we strengthen the condition on \(\gamma\) to \(\gamma\geq d\) (so that the \emph{a priori} bound
of Theorem \ref{thm:a-priori-bound} is available) then
there is a kind of compactness result for \(\eps\)-near-sequential-\(\Pi\)-paths.

\begin{thm}\label{thm:near-sequential-compactness}
Suppose that \(\gamma\geq d\geq2\) and \(T<\infty\).
For \(n=1\), \(2\), \ldots, let \(\widetilde{\xi}_n:[0,T]\to\Reals^d\) be an
\(\eps_n\)-near-sequential-\(\Pi\)-path from \(x\) to \(y\), with \(\eps_n\downarrow0\).
Then there are subsequences \(\{\widetilde{\xi}_{n_k}: k=1, 2, \ldots\}\) which converge uniformly 
to \(\Pi\)-path limits.
\end{thm}

\begin{proof}
Since \(\eps_n\) is decreasing in \(n\), 
each \(\eps_n\)-near-sequential-\(\Pi\)-path \(\widetilde{\xi}_n\) obeys the single modified speed-limit \(\max\{\eps_1,V\}\).
Hence the comparison argument of Theorem \ref{thm:a-priori-bound} can be adapted to show that all the 
\(\eps_n\)-near-sequential-\(\Pi\)-paths \(\widetilde{\xi}_1\), \(\widetilde{\xi}_2\), \ldots lie in a single ball \(B\)
of radius \(R\) depending on \(V\) and \(\eps_1\).

Consequently the \(\widetilde{\xi}_1\), \(\widetilde{\xi}_2\), \ldots obey a uniform Lipschitz condition (with Lipschitz
constant given by the speed of the fastest line to hit the ball \(B\)); therefore by the Arzela-Ascoli theorem we can extract a uniformly convergent
subsequence \(\{\widetilde{\xi}_{n_k}: k=1, 2, \ldots\}\) whose limit \(\widetilde{\xi}_\infty\) is also a Lipschitz path
with the same Lipschitz constant.

The persistence of Lipschitz constants in the limit also holds locally. 
For fixed \(\lambda>0\), consider the open set \(\Silhouette_v^c=\{x:V(x)<v\}\). 
Fix \(0<s<t<T\) such that \(\Image(\widetilde{\xi}_\infty|_{[s,t]})\subset\Silhouette_v^c\).
But \(\widetilde{\xi}_\infty\) is continuous, so \(\Image(\widetilde{\xi}_\infty|_{[s,t]})\) is compact; therefore the uniform convergence
of \(\widetilde{\xi}_{n_k}\to\widetilde{\xi}_\infty\) implies that for all \(k\geq k_\lambda\) we have
\[
 \Image(\widetilde{\xi}_n|_{[s,t]})\quad\subset\quad \Silhouette_v^c\,.
\]
It follows that if \(k\geq k_\lambda\) then \(\widetilde{\xi}_{n_k}\) satisfies a \(\Lipschitz(\max\{\eps_{n_k}, v\})\) condition
over the time set \([s,t]\).
Bearing in mind that \(\eps_n\downarrow0\), 
we deduce that \(\widetilde{\xi}_\infty\) satisfies
a \(\Lipschitz(v)\) condition whenever \(\widetilde{\xi}_\infty\) belongs to \(\Silhouette_v^c\).
This implies that the following is a Lebesgue-null subset of \([0,T]\):
\[
 \{t\in[0,T]: |\widetilde{\xi}_\infty'(t)|> v \text{ and } \widetilde{\xi}_\infty(t)\not\in\Silhouette_v\}\,.
\]
Thus the subsequential limit \(\widetilde{\xi}_\infty\) is a \(\Pi\)-path (Definition \ref{def:pi-path}).
%
%
\end{proof}

This allows us to deduce the measurability of the random variable which is given by the time taken for a \(\Pi\)-geodesic to move between
specified end-points \(x_1\) and \(x_2\).

\begin{cor}\label{cor:measurable-time}
 Suppose that \(\gamma>d\). Fix \(x_1\) and \(x_2\) in \(\Reals^d\), and let \(T\) be the least time such that there is a \(\Pi\)-path running from \(x_1\) to \(x_2\) in time \(T\), equivalently, 
 such that the (possibly non-unique) \(\Pi\)-geodesic from \(x_1\) to \(x_2\) has duration \(T\).
Then \(T\) is a function of the marked line process \(\Pi\): it is in fact measurable and hence a random variable.
\end{cor}
\begin{proof}
Consider the event \(E_{\eps,\tau}\) that there is an \(\eps\)-near-sequential-\(\Pi\)-path from \(x_1\) to \(x_2\) of duration at most \(\tau\).
This event is measurable,
because we may restrict attention to a countable sub-family of \(\eps\)-near-sequential-\(\Pi\)-paths, determined for example
so that constituent line segments are bounded by the intersection point process of \(\Pi\).

But it is a consequence of the above results that
\[
 \bigcap_\eps E_{\eps,\tau} \quad=\quad [\text{ duration of \(\Pi\)-geodesic from \(x_1\) to \(x_2\) is no more than \(\tau\) }]\,.
\]
For Theorem \ref{thm:pi-path-approximation} shows that the existence of such a \(\Pi\)-geodesic leads to the construction
of \(\eps\)-near-sequential-\(\Pi\)-paths from \(x_1\) to \(x_2\) of the same duration as the \(\Pi\)-geodesic. On the other hand
it is an immediate consequence of
Theorem \ref{thm:near-sequential-compactness} that if \(E_{\eps_n,\tau}\) is non-empty for a sequence \(\eps_n\downarrow0\)
then there must exist a \(\Pi\)-path from \(x_1\) to \(x_2\) of duration \(\tau\). (Note that
we may prolong the duration of any \(\eps\)-near-sequential-\(\Pi\)-path simply by holding it at its destination.)
Finally, the events \(E_{\eps,\tau}\) are decreasing in \(n\), so the intersection \(\bigcap_\eps E_{\eps,\tau}\)
can be reduced to a countable intersection. It follows that 
\[
\left[\text{ duration of \(\Pi\)-geodesic from \(x_1\) to \(x_2\) is no more than \(\tau\) }\right]
\]
is a measurable event.
\end{proof}
We note that simple selection criteria can be used to establish measurable maps which yield intervening \(\Pi\)-geodesics for each
pair of end-points \(x_1\) and \(x_2\).

\section[\texorpdfstring{$\Pi$}{Π}-geodesics: a.s. uniqueness in d=2]{\(\Pi\)-geodesics: almost-sure uniqueness in dimension \(2\)}\label{sec:pi-geodesics-uniqueness}
In the simplest non-trivial case, namely \(d=2\), it can be shown that the \(\Pi\)-geodesic between two specified points
is almost surely unique. 
The method of proof makes essential use of the point-line duality of the plane, so
will not extend to the case \(d\geq3\).

\begin{rem}\label{rem:non-uniqueness}
 The assertion of almost sure uniqueness between 
of \(\Pi\)-geodesic connections between two specified points does \emph{not} imply that almost surely all pairs of points are connected by 
unique \(\Pi\)-geodesics: a simple counterexample can be constructed by considering the possibility that three lines, of speeds only just below unit speed,
form an approximate equilateral triangle \(\Delta\)
of near unit-side length. Let \(\rho\) be the perimeter of \(\Delta\), and suppose all other lines within \(\rho/4\) of \(\Delta\) are of less than speed \(1/2\), while all lines hitting
the interior of
\(\Delta\) are of speed substantially less than \(1/2\). (How much less depends on the approximation to equilateral shape.) Both these events have positive probability.
Then consider the two \(\Pi\)-paths of length \(\rho/2\) running either way round \(\Delta\) from a given reference point on the boundary of \(\Delta\). These form two distinct \(\Pi\)-geodesics
between the same end-points. (The construction is illustrated in Figure \ref{fig:non-unique}.)
\end{rem}
Note that structures similar to this counterexample will exist at all scales in the random metric space produced from \(\Reals^2\) furnished with a random metric derived from \(\Pi=\Pi^{(2,\gamma)}\) for \(\gamma>2\).
So these random metric spaces are far from hyperbolic in the sense of Gromov (defined for example in \citealp[\S8.4]{BuragoBuragoIvanov-2001}).
\begin{Figure}
  \includegraphics[width=2.5in]{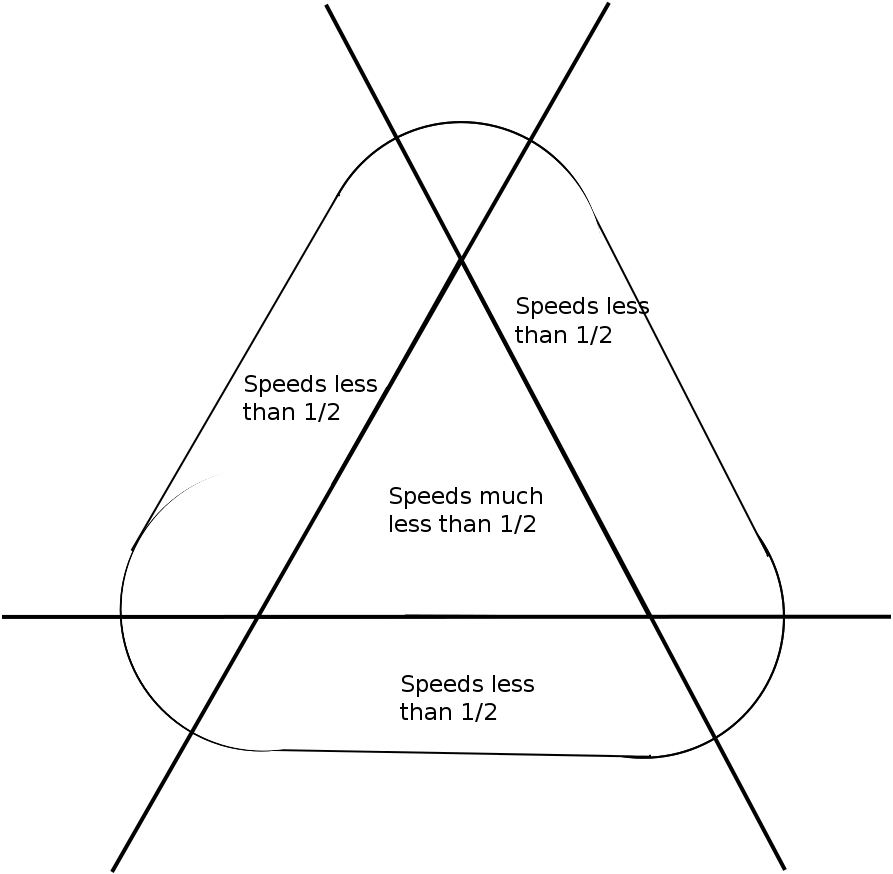}
\centering
\caption{\label{fig:non-unique}
Within the approximate equilateral triangle delineated by three fast lines, speeds are slow enough to prevent short-cuts.
Outside the triangle, up to a distance of one quarter of the perimeter, speeds are slow enough that there is no possibility of \(\Pi\)-geodesics
looping outside this region and then returning to the triangle.
}
\end{Figure}

We begin with a structural result about \(\Pi\)-geodesics in dimension \(2\), namely that if \(\ell\) is a line from \(\Pi\) (hence of positive speed-limit)
which contributes a segment to a \(\Pi\)-geodesic \(\xi\) then \(\xi\) joins and leaves \(\ell\) using simple intersections of \(\ell\) with other lines in \(\Pi\).
\begin{defn}[Proper encounter of a line by a \(\Pi\)-path]\label{def:proper-encounter}
 Suppose \(d=2\) and \(\gamma>2\). We say that a \(\pi\)-path \(\xi\) \emph{encounters a line \(\ell\) of \(\Pi\) properly}
at \(\xi(t)\in\ell\) if there is \(\varepsilon>0\) such that within \(\ball(\xi(t),\varepsilon)\) the \(\Pi\)-path \(\xi\) is not contained solely in \(\ell\), 
but lies in the union of \(\ell\) and a further line \(\widetilde\ell\) 
from \(\Pi\).
\end{defn}
The notion of a proper encounter 
is vacuous for \(\Pi\)-paths in the case \(d>2\), because then almost surely lines of the Poisson line process \(\Pi\) do not intersect each other.

\begin{thm}\label{thm:encounters}
 Suppose \(d=2\) and \(\gamma>2\). 
With probability \(1\), for each line \(\ell\in\Pi\) and each \(\Pi\)-geodesic \(\xi\), the intersections of \(\xi\)
with \(\ell\) form a finite disjoint union of intervals, such that the non-singleton intervals are delimited by proper encounters of \(\xi\) with \(\ell\).
\end{thm}
\begin{proof}
First note that with probability \(1\) there are no triple intersections of lines \(\ell_1\), \(\ell_2\), \(\ell_3\) from \(\Pi\).
Given this, the remainder of the argument is purely geometric, and therefore applies to all \(\Pi\)-geodesics simultaneously.

Consider the set of all lines \(\ell\) in \(\Pi\) intersecting \(\xi\). As noted in Remark \ref{rem:geodesic-simplification},
the \(\Pi\)-geodesic property implies that
the intersection of \(\xi\) with the fastest such line must be a single (possibly trivial) interval. 
This is because the fastest route between first and last intersection with this fastest line must lie along the line.
For the second fastest line, the intersection must be formed as the union of at most two
intervals, which must be encountered respectively before and after the encounter with the fastest line. Continuing this argument, the 
intersection of \(\xi\) with the \(k^\text{th}\) fastest line must be the union of at most \(2^{k-1}\) intervals.
Thus for \emph{any} line \(\ell\) from \(\Pi\),
if \(\xi\) intersects \(\ell\) at all then the intersection set must be the union of a finite number of intervals, some of which may be trivial.

For a given \(\Pi\)-geodesic \(\xi\), fix a given line \(\ell_1\) from \(\Pi\) which intersects \(\xi\), and consider the start \(\xi(t)\) of a
non-singleton intersection interval \([\xi(t),\xi(s)]\) between \(\xi\) and \(\ell_1\). 
(Reversing time, the following argument will apply to the departure point \(\xi(s)\) as well as to the arrival point \(\xi(t)\).)
For sufficiently small \(\varepsilon\), either \(\ell_1\) will be the fastest line in \(\ball(\xi(t),\varepsilon)\), or it will be the second fastest, and the fastest line \(\ell_2\)
will intersect \(\ell_1\) at \(\xi(t)\). Choose \(u<t\) to be as small as possible subject to the requirement that \(\xi|_{(u,t]}\subset \ball(\xi(t),\eps)\). It follows from the \(\Pi\)-geodesic property that if \(\xi|_{(u,t)}\) hits \(\ell_2\) at time \(v\in(u,t)\) (assuming \(\ell_2\) exists) then \(\xi|_{[v,t]}\) must run along \(\ell_2\), so that \(\xi\) makes a proper encounter with \(\ell_1\) using the line \(\ell_2\). On the other 
hand, if \(\xi|_{(u,t)}\) does not hit \(\ell_2\) then it cannot hit \(\ell_1\), since if it did so at time \(v\in(u,t)\) then \(\xi|_{[v,t]}\) would have to run along \(\ell_1\), contradicting the maximality of \([t,s]\). Thus we may restrict attention to the case when \(\xi|_{(u,t)}\) hits neither \(\ell_1\) nor \(\ell_2\). This and following features of the construction are illustrated in Figure \ref{fig:encounter}.
\begin{Figure}
  \includegraphics[width=2.5in]{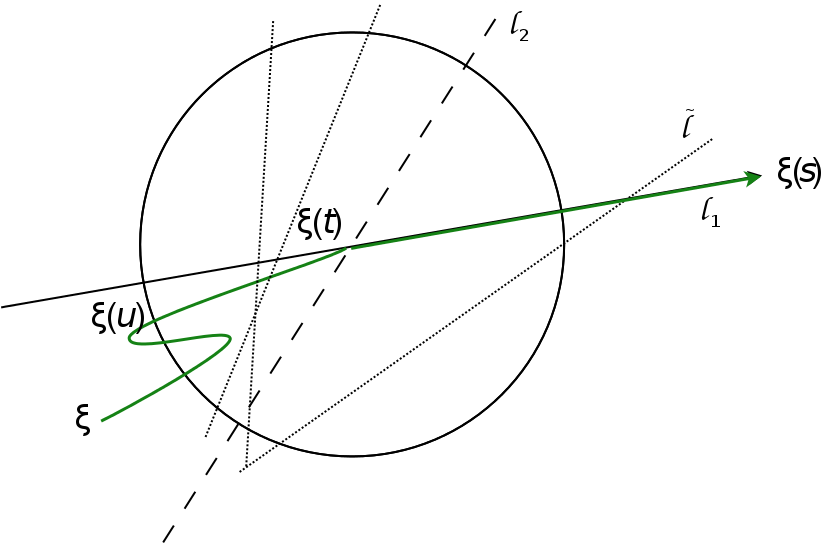}
\centering
\caption{\label{fig:encounter}
Illustration of the construction used in the proof of Theorem \ref{thm:encounters} to demonstrate that a geodesic \(\xi\) must make a proper encounter on
a line \(\ell_1\in\Pi\). Here \(\ell_2\) is a possible faster line; by making the enclosing ball small
enough we can exclude the possibility that such a faster \(\ell_2\) hits \(\ell_1\) in any place other than the
point \(\xi(t)\) where \(\xi\) hits \(\ell_1\) for the first time. Dotted lines are lines of low cost, where the notion of cost is described in the
construction.}
\end{Figure}


%

We introduce the notion of \emph{cost}, based on comparison between the motion defined by \(\xi\) over \((u,t)\) (say) 
and the motion defined by a comparison particle \(\widetilde{\xi}\), which begins at time \(u\) at the location which is the 
projection of \(\xi(u)\) on \(\ell_1\), and which continues at maximum speed (\(w\), say) along \(\ell_1\) in the direction from \(\xi(t)\) to \(\xi(s)\). 
We compute the cost of following \(\xi\) rather than \(\widetilde{\xi}\) 
in terms of the time by which \(\widetilde{\xi}\) leads \(\xi\) when \(\xi\) hits \(\ell_1\) (namely, at time \(t\)). This is given by the integral
\[
 \frac{1}{w}\int_{u}^t(w - |\xi'(a)|\cos\theta(a))\d{a}\quad=\quad \frac{1}{w}\int_{u}^t(w - v(a)\cos\theta(a))\d{a}\,,
\]
where \(\theta(a)\) is the angle that \(\xi'(a)\) makes with \(\ell\),
and setting
\(v(a)=|\xi'(a)|\). (Note that \(v(a)=|\xi'(a)|=V(\xi(a))\), because \(\xi\) is a \(\Pi\)-geodesic). 

Let \(H\) be the perpendicular distance between \(\xi(u)\) and \(\ell_1\). We re-parametrize in terms of the perpendicular distance between \(\xi(a)\) and \(\ell_1\),
removing those parts of the integral for which \(\xi'(a)\) is not directed towards \(\ell_1\) (in which case the contribution to the integral is certainly positive, since \(\ell_1\) is faster than any other line
used by \(\xi\) over \((u,t)\)).
Setting 
\[
\overline{a}(h)
\quad=\quad
\inf\left\{a:\text{ perpendicular distance of }\xi(a)\text{ from }\ell_1\text{ is }h\right\}\,,
\]
and \(\overline{v}(h)=v(\overline{a}(h))\),
and \(\overline{\theta}(h)=\theta(\overline{a}(h))\), we find
\begin{multline*}
 \frac{1}{w}\int_{u}^t(w - |\xi'(a)|\cos\theta(a))\d{a}\quad\geq\quad \frac{1}{w}\int_{0}^H(w - \overline{v}(h)\cos\overline{\theta}(h))\frac{\d{h}}{\overline{v}(h)\sin\overline{\theta}(h)}\\
\quad=\quad
\int_{0}^H\left(
\frac{\csc\overline{\theta}}{\overline{v}}
 - \frac{\cot\overline{\theta}}{w}
\right)\d{h}
\,.
\end{multline*}
Accordingly, define the relative \emph{cost index} of a given line \(\ell\) from \(\Pi\) 
(compared with \(\ell_1\)) by
\begin{equation}\label{eqn:cost-index}
 c(\ell)\quad=\quad \frac{\csc{\theta}}{{v}}
 - \frac{\cot{\theta}}{w}\,,
\end{equation}
where \(v\) is the speed-limit of \(\ell\), and \(\theta\) is the angle it makes with \(\ell_1\). Evidently 
the time by which \(\widetilde{\xi}\) leads \(\xi\) when \(\xi\) hits \(\ell_1\) can be controlled
in terms of an integral of cost indices of lines along which \(\xi\) travels when directed towards \(\ell_1\).
The cost index of \(\ell_1\) is not defined, though a limiting argument gives the value \(0\). Note too that, for any line \(\ell\) of speed \(v\), 
the cost index of \(\ell\) turns out to be positive if \(v< w\).

Consider line-space parametrized using \(\xi(t)\) and \(\ell_1\) as reference point and reference line,
restricting attention to lines with speed-limit less than \(w\) (the speed-limit for \(\ell_1\)). The intensity measure \(\tfrac{\gamma-1}{2} \,v^{-\gamma}\d{v}\,\d{r}\, \d{\theta}\)
may be re-expressed in terms of \(c\), \(r\), and \(\theta\): since
\[
 \frac{\d{c}}{\d{v}}\quad=\quad - \frac{\csc\theta}{v^2}\,,
\]
it follows that in the new coordinates the intensity measure is 
\begin{equation}\label{eq:cost-measure}
 \frac{\gamma-1}{2} \,\sin\theta\left(c \sin\theta + \frac{\cos\theta}{w}\right)^{\gamma-2} \,\d{c}\,\d{r}\,\d{\theta}\,.
\end{equation}
Now \(v^{-1}=c \sin\theta +\tfrac{\cos\theta}{w}>0\), so the measure determined by \eqref{eq:cost-measure}
is non-negative.
Consider the line pattern of lines with cost smaller than a specified threshold \(c_0\). 
From the form of \eqref{eq:cost-measure}, this pattern is locally finite. On the other hand, the line pattern of lines
with cost exceeding a specified threshold must be locally dense even when constrained by 
requiring
angle \(\theta\) to lie within a small interval.

We pick \(\widetilde\ell\) to be the lowest cost line separating \(\xi(s)\) from \(\xi_{(u,t)}\)
(see Figure \ref{fig:encounter} for a possible configuration, notice that this line may or may not be \(\ell_2\)),
and we
determine the minimal \(v\in[u,t]\) such that all the lines involved in \(\xi|_{(v,t)}\)
are more expensive than \(\widetilde\ell\).

If \(v=t\) then consider the \(\liminf\) as \(\delta\to0\) of the costs of lines involved in \(\xi|_{(t-\delta,t)}\).
This must be finite, for otherwise a low-cost line would produce a path faster than the \(\Pi\)-geodesic. 
Since there are only finitely many low-cost lines near \(\xi(t)\),
this means there must be at least one low-cost line which is repeatedly visited by \(\xi\)
in every interval \((t-\delta,t)\); so this line must pass through \(\xi(t)\) and therefore must either be \(\ell_1\)
(which we have excluded) or \(\ell_2\) (which is a case already disposed of). Hence we may suppose 
\(v<t\).

If \(v<t\), pick the line \(\ell^*\) of least cost which is hit by \(\xi|_{(v,t)}\). 
Suppose (a) this hits the component of \(\ell_1\setminus\xi(s)\) not containing \(\xi(t)\). Then a combination of \(\ell^*\) and \(\widetilde\ell\) and \(\ell_1\) provides a faster way to get to \(\xi(s)\) than is provided by \(\xi\), again violating the \(\Pi\)-geodesic property of \(\xi\).
Otherwise (b) this line \(\ell^*\) does not hit the component of \(\ell_1\setminus\xi(s)\) not containing \(\xi(t)\). 
If \(\xi\) does not meet \(\ell_1\) at \(\xi(t)\) using \(\ell^*\), then \(\ell^*\) 
followed by \(\ell_1\) provides a faster way to get to \(\xi(s)\) than is provided by \(\xi\), violating the \(\Pi\)-geodesic property of \(\xi\). 

It follows that if \(\xi\) does not meet \(\ell_1\) at \(\xi(t)\) using \(\ell_2\) then 
it must meet \(\ell_1\) at \(\xi(t)\) using \(\ell^*\), proving the result.

%

As noted above, a time-reversal argument deals with the departure time \(s\). Consequently all the countably many non-singleton intersections of \(\xi\) with lines of positive speed-limit must be proper.
\end{proof}

Note in passing that in higher dimension \(d>2\) the quantity analogous to the cost \eqref{eqn:cost-index} varies along each line.

We can now prove almost sure uniqueness of \(\Pi\)-geodesics between specified pairs of points in the planar case.

\begin{thm}\label{thm:uniqueness}
 Suppose \(\gamma>d=2\). Consider two points \(x\) and \(y\) in the plane \(\Reals^2\). Almost surely there is just one \(\Pi\)-geodesic connecting
\(x\) and \(y\).
\end{thm}
\begin{proof}
First, note the following consequence of Theorem \ref{thm:a-priori-bound}: as \(R\to\infty\), so 
\[
 \Prob{\text{ all geodesics from \(x\) to \(y\) are contained in }\ball(\origin,R) }\quad\to\quad 1\,.
\]

Furthermore, Theorem \ref{thm:encounters} implies that the following assertion holds almost surely: 
all \(\Pi\)-geodesics join or leave any line \(\ell\) in \(\Pi\) at only countably many possible places, namely the intersection
points of \(\ell\) with the rest of \(\Pi\). Moreover any particular \(\Pi\)-geodesic joins or leaves any particular line at only finitely many of these places. 

Finally, note that if two \(\Pi\)-geodesics from \(x\) to \(y\) intersect at \(z\) then they must do so at the same relative time as measured from \(x\).

Bearing these facts in mind, we now develop the proof.

For fixed \(v>0\), pick a line \(\ell_0\) uniformly at random from \(\Pi_{v_0}\) such that \(\ell_0\hits\ball(\origin,R)\). 
Note that the speed \(V\) of \(\ell_0\) has a Pareto distribution, with density \((\gamma-1) (v/v_0)^{-\gamma}\) for \(v>v_0\).
Note too that \(V\) is independent of the physical location of \(\ell_0\) and (by Slivnyak's theorem) is independent of
\(\Pi\setminus\{\ell_0\}\) which itself is a statistical copy of \(\Pi\).
Then (almost surely) for all sufficiently large \(R>0\) we know all \(\Pi\)-geodesics from \(x\) to \(y\) belong to the path-set
\(\mathcal{P}^{\ell_0}(x,y)=\bigcup_{u>0}\mathcal{P}^{\ell_0}_{u}(x,y)\), where \(\mathcal{P}^{\ell_0}_{u}(x,y)\) is the set of \(\Pi\)-paths from \(x\) to \(y\) lying in \(\ball(\origin,R)\) for which:
\begin{itemize}
 \item the \(\Pi\)-path is always run at maximal permissible speed;
 \item excursions of the \(\Pi\)-path away from \(\ell_0\) are \(\Pi\)-geodesics;
 \item the intersections of the \(\Pi\)-path with \(\ell_0\) form a finite disjoint union of intervals \([a,b]\);
 \item and from these intervals the non-singleton intervals are delineated by proper encounters of the \(\Pi\)-path and \(\ell_0\),
 moreover these end-points lie in \(\ell_0\cap\bigcup\{\ell:\ell\in\Pi_u\}\).
\end{itemize}

We further decompose \( \mathcal{P}^{\ell_0}_{u}(x,y)=\bigcup_\sigma \mathcal{P}^{\ell_0}_{u}(x,y;\sigma)\), where \(\sigma\) ranges over the family of finite sequences of pairs of (signed) integers, such that the closed intervals delineated by the pairs of integers are disjoint. To define \(\mathcal{P}^{\ell_0}_{u}(x,y;\sigma)\), consider the doubly-infinite point sequence \(\ell\cap\bigcup\{\ell:\ell\in\Pi_u\}\), and index the points by \(\mathbb{Z}\) once and for all, using a fixed sense of direction along \(\ell_0\) and arranging for the interval determined by the points indexed by \(0\) and \(1\) to be the (almost surely unique) interval nearest to \(\origin\). Then \(\mathcal{P}^{\ell_0}_{u}(x,y;\sigma)\)
is composed of those \(\Pi\)-paths in \(\mathcal{P}_{u}(x,y)\) for which the union of disjoint non-singleton intervals of intersection with \(\ell_0\) equals the union of the intervals bounded by pairs of points indexed by the pairs of \(\sigma\), moreover \(\sigma\) lists these intervals in the order in which they are visited and according to the direction in which they are travelled. 

It is a consequence of the defining properties of \(\mathcal{P}^{\ell_0}(x,y)\) \emph{etc}, and the property that intersecting 
\(\Pi\)-geodesics from \(x\) to \(y\) must visit their intersections at the same relative times, that all \(\Pi\)-paths in \(\mathcal{P}^{\ell_0}_{u}(x,y;\sigma)\)
spend the same amount of time \(S_\sigma\) outside \(\ell_0\). 
Moreover, consider the lengths \(L_{\sigma_{1}}\), \(L_{\sigma_{2}}\), \ldots, \(L_{\sigma_{k}}\) corresponding to the intervals bounded by pairs of points indexed by the elements of \(\sigma=(\sigma_1, \sigma_2, \ldots, \sigma_k)\). These are 
sums of independent Gamma random variables of the same rate, and all \(\Pi\)-paths in \(\mathcal{P}^{\ell_0}_{u}(x,y;\sigma)\)
spend the same amount of time \(L_\sigma/V=(L_{\sigma_{1}}+L_{\sigma_{2}}+\ldots+L_{\sigma_{k}})/V\) on \(\ell_0\).
Moreover the \(S_\sigma\) and the \(L_{\sigma_i}\) random variables are statistically independent of the speed \(V\) of \(\ell_0\).

If \(\sigma\neq\widetilde{\sigma}\) then the sums \(L_\sigma=L_{\sigma_{1}}+L_{\sigma_{2}}+\ldots+L_{\sigma_{k}}\), \(L_{\widetilde{\sigma}}=L_{\widetilde{\sigma}_{1}}+L_{\widetilde{\sigma}_{2}}+\ldots+L_{\widetilde{\sigma}_{\widetilde{k}}}\)  can be decomposed into summands over shared or distinct Gamma random variables to reveal that \(L_\sigma-L_{\widetilde{\sigma}}\) has a non-degenerate probability density whenever \(\sigma\neq\widetilde{\sigma}\), and therefore \(\Prob{L_\sigma=L_{\widetilde{\sigma}}}=0\). Hence \(\Pi\)-paths from \(\mathcal{P}^{\ell_0}_{u}(x,y;\sigma)\) have a common duration of \(S_\sigma+L_\sigma/V\), and 
\(\Pi\)-paths from \(\mathcal{P}^{\ell_0}_{u}(x,y;\widetilde{\sigma})\) have a common duration of \(S_{\widetilde{\sigma}}+L_{\widetilde{\sigma}}/V\), and the Pareto distribution and independence of \(V\) implies that
\[
\Prob{S_\sigma+L_\sigma/V=S_{\widetilde{\sigma}}+L_{\widetilde{\sigma}}/V}=0 \text{ if }\sigma\neq\widetilde{\sigma}\,. 
\]
Thus almost surely, for all the countably many different \(\sigma\neq\widetilde{\sigma}\), the common durations of \(\Pi\)-paths from \(\mathcal{P}^{\ell_0}_{u}(x,y;\sigma)\) and \(\mathcal{P}^{\ell_0}_{u}(x,y;\widetilde{\sigma})\) are different.

It follows that almost surely all the \(\Pi\)-geodesics between \(x\) and \(y\) must traverse the same set of non-singleton intervals in the same direction along \(\ell_0\), since any two such \(\Pi\)-geodesics will have to belong to the same \(\mathcal{P}^{\ell_0}_{u}(x,y;\sigma)\) for sufficiently small \(u>0\). But this must then hold for all \(\ell\in\Pi\), and therefore (since \(\Pi\)-geodesics  
must intersect at the same time as measured from \(x\)) almost surely all \(\Pi\)-geodesics between \(x\) and \(y\) must agree.
\end{proof}

The argument here is delicate: for example it is \emph{not} the case that the set of lengths along lines between intersections is linearly independent 
if we consider the ensemble of lengths
of a unit Poisson line process. Indeed the tessellation produced by a Poisson line process will be rigid; consideration of various triangles shows that 
the length of any single segment will be determined
by the lengths of all the other segments, so long as the incidence geometry of the segments is known.

The almost-sure uniqueness of planar \(\Pi\)-geodesics implies that planar
spatial networks formed from the Poisson line process model satisfy
property \ref{def:SIRSN-item-route} of Definition \ref{def:SIRSN}.

\section[\texorpdfstring{$\Pi$}{Π}-geodesics: finite mean-length in d=2]{\(\Pi\)-geodesics: finiteness of mean-length in dimension \(2\)}\label{sec:pi-geodesics-finite-mean-length}
One might conceive that a \(\Pi\)-geodesic between two fixed points might be of finite length but not of finite mean length.
However this is not the case, at least in dimension \(d=2\). We begin to show this by first establishing the finite mean length of constrained \(\Pi\)-geodesics, {restricted} to lie within specified
(two-dimensional) balls.
\begin{lemma}\label{lem:finite-mean-length-in-ball}
 Suppose \(d=2\) and \(\gamma>2\). Consider \(x_1\), \(x_2\in\Reals^2\), fix \(r_0>|x_2-x_1|\), and consider
the least time by which it is possible to connect \(x_1\) to \(x_2\) by a \(\Pi\)-path which remains entirely within \(\ball(\tfrac{x_1+x_2}{2},r_0)\):
\begin{multline*}
 T_{r_0}^* \quad=\quad \inf\Big\{T \;:\; \text{ there is }\xi\in\Paths_T \text{ such that } \\
 \xi(0)=x_1, \; \xi(T)=x_2, \; \text{ and }\xi(t)\in\ball(\tfrac{x_1+x_2}{2},r_0)\text{ for all }t\in[0,T] \Big\}\,.
\end{multline*}
Then \(\Expect{T_{r_0}^*\;|\: V_0}<\infty\), where \(V_0\) is the speed-limit of the fastest line hitting \(\ball(\tfrac{x_1+x_2}{2},r_0)\);
moreover the following finite expectation provides an upper bound on the mean length of a \(\Pi\)-path connecting \(x_1\) to \(x_2\) within \(\ball(\tfrac{x_1+x_2}{2},r_0)\):
\[
 \Expect{V_0 T_{r_0}^*}\quad<\quad \infty\,.
\]
\end{lemma}
\begin{proof}
Because we work only in dimension \(2\),
and seek an upper bound on \(\Pi\)-geodesic length, we are able to concentrate on \(\Pi\)-paths defined by joining together sequences of line segments;
there is no need to
negotiate the complexities of the tree construction described in Theorems \ref{thm:connection} and \ref{thm:metric-space}.
The time taken by such a \(\Pi\)-path, constrained to lie within \(\ball(\tfrac{x_1+x_2}{2},r_0)\) and connecting \(x_1\) to \(x_2\), necessarily provides an upper bound 
on the \(\ball(\tfrac{x_1+x_2}{2},r_0)\)-constrained \(\Pi\)-geodesic connecting \(x_1\) to \(x_2\). Thus the finiteness of \(\Expect{V_0 T_{r_0}^*}\), together with the fact that \(V_0\) 
is the maximum speed attainable within \(\ball(\tfrac{x_1+x_2}{2},r_0)\), provides an upper bound on the mean length of the 
constrained \(\Pi\)-geodesic connecting \(x_1\) to \(x_2\) which is restricted to lie within \(\ball(\tfrac{x_1+x_2}{2},r_0)\).

Suppose that \(|x_2-x_1|=r_1<\half r_0\). 
Without loss of generality, we suppose that \(x_1+x_2=\origin=(0,0)\), and \(x_1=(-\half r_1,0)\), \(x_2=(\half r_1,0)\).
We shall join \(x_1\) and \(x_2\) together by working towards the two points by two paths commencing on the line segment \(\sigma_1=\{(0,h):0\leq h\leq r_1\}\);
we will then be able to join the two \(\Pi\)-paths together by prolonging the first line segment used in the construction of one of the \(\Pi\)-paths.

The constructions of the two \(\Pi\)-paths are entirely similar, so we shall focus on the \(\Pi\)-path leading to \(x_1\).

Suppose that the fastest line intersecting \(\ball(\origin,r_0)\) has speed-limit \(V_0\). Exploiting the notion of meta-slowness described above
in the proof of Theorem \ref{thm:a-priori-bound}, we know that
if \(S_0=V_0^{-(\gamma-1)}\) is the meta-slowness of this line then 
\begin{equation}\label{eq:fastest-line}
 S_0 \quad=\quad \frac{1}{\pi r_0} E_0 \qquad \text{ where \(E_0\) is distributed as }\text{Exponential}(1)\,.
\end{equation}
The following formulae are simplified if 
our \(\Pi\)-path constructions are required to avoid using this line. The first line used in the \(\Pi\)-path running to \(x_1\) will be the fastest line \(\ell_1\) with speed less than \(V_0\)
and intersecting both the line segment \(\sigma_1\) and the line segment from \(x_1\) to \(\tfrac{3}{2}x_1\). 
Suppose that the speed-limit of this line is \(V_1\), so the meta-slowness is \(S_1=V_1^{-(\gamma-1)}\).
We use standard integral geometry of lines and Pythagoras to show that the line measure of all such lines is
\[
 \half\left(|x_1-(0,r_1)|+|\tfrac{3}{2}x_1|-|\tfrac{3}{2}x_1-(0,r_1)|-|x_1\|\right)
\quad=\quad 
\frac{\sqrt{5}-2}{4} \; r_1\,.
\]
Consequently we may deduce
\[
  S_1 \quad=\quad S_0 + \frac{4}{(\sqrt{5}-2) r_1} E_1 \qquad \text{ where \(E_1\) is distributed as }\text{Exponential}(1)\,,
\]
and \(E_1\) is independent of \(S_0\) (equivalently \(V_0\)) and the geometry of the line \(\ell_1\).

\begin{Figure}
  \includegraphics[width=2in]{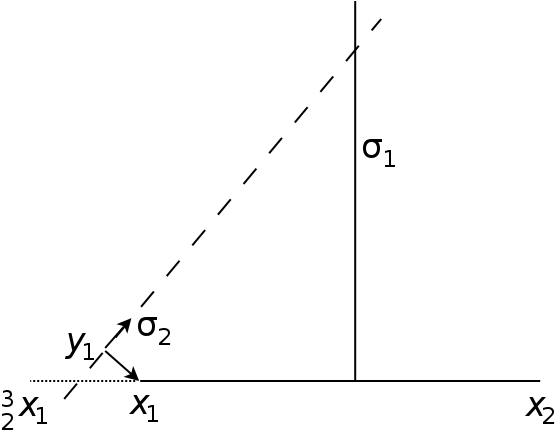}
\centering
\caption{\label{fig:mean-length}
Successive construction of one of two components of a \(\Pi\)-path from \(x_1\) to \(x_2\).
}
\end{Figure}
Let \(y_1\) be the point on \(\ell_1\) closest to \(x_1\), and note that the distance from \(y_1\) to \(\sigma_1\) along \(\ell_1\)
is bounded above by the distance from \(\tfrac{3}{2}x_1\) to \((0,r_1)\), namely
\[
 \sqrt{\frac{9}{16}r_1^2 + r_1^2} \quad=\quad \frac{5}{4}\, r_1\,.
\]
The construction is illustrated in Figure \ref{fig:mean-length}.

This construction is continued recursively, for example replacing the origin by the point \(y_1\) on \(\ell_1\) closest to \(x_1\),
\(r_0\) by \(r_1\),
and replacing the segment \(\sigma_1\) by a segment \(\sigma_2\) begun at \(y_1\), directed along \(\ell_1\) towards the start of the \(\Pi\)-path,
of length \(r_2=2|y_1-x_1|\). Simple geometric arguments show that both \(|y_1-x_1|\) and \(|y_1-\tfrac{3}{2}x_1|\) are bounded above by \(\tfrac{1}{4}r_1\), so 
the distance that \(\sigma_2\) extends 
from \(\tfrac32 x_1\) cannot exceed \(\tfrac{3}{4}r_1\), while the distance between \(\tfrac32 x_1\) and \(\sigma_1\) is \(\tfrac34 r_1\).
This construction can be used to generate a new line \(\ell_2\), of meta-slowness \(S_2=V_2^{-(\gamma-1)}\) 
which is required to be strictly greater than \(S_1\),
and a new closest distance \(r_2/2\) from \(x_1\) to \(\ell_2\). The calculations show that \(r_2\leq r_1/2\).

In general the \(n^\text{th}\) line \(\ell_n\) of the construction has meta-slowness \(S_n=V_n^{-(\gamma-1)}\)
with
\begin{equation}\label{eqn:recursion}
 S_n \quad=\quad S_{n-1} + \frac{4}{(\sqrt{5}-2) r_n} E_n \qquad \text{ where \(E_n\) is distributed as }\text{Exponential}(1)\,,
\end{equation}
and \(T_1\) is independent of \(S_0\), \ldots, \(S_{n-1}\) (equivalently \(V_0\), \ldots, \(V_{n-1}\)) and the geometry of the lines \(\ell_1\), \ldots, \(\ell_{n-1}\).
Here \(r_n\) is the closest distance from \(x_1\) to \(\ell_n\), and \(r_n<r_{n-1}/2\); the length of the new segment 
(running from \(\sigma_n\) to \(y_n\) along \(\ell_n\)) is bounded above by \(\tfrac{5}{4}r_n\).

Evidently we have constructed a \(\Pi\)-path from \(\sigma_1\) to \(x_1\), built as a sequence of line segments.
Total time of travel is bounded above by
\begin{multline*}\label{eqn:time-of-travel}
 \sum_{n=1}^\infty S_n^{\frac{1}{\gamma-1}}  \times \frac{5}{4} r_n
\quad=\quad
\frac{5}{4} 
\sum_{n=1}^\infty \left(
S_0+ \frac{4}{\sqrt{5}-2}\left(\frac{E_1}{r_1}+\ldots+ \frac{E_n}{r_n}\right)\right)^{\frac{1}{\gamma-1}}   r_n
\quad\leq\quad
\\
\frac{5r_0^{\frac{\gamma-2}{\gamma-1}}}{4} 
\sum_{n=1}^\infty \left(
2^{-(n-1)}r_0S_0+
\frac{4}{\sqrt{5}-2}\left(
 2^{-(n-1)}{E_1}+ 2^{-(n-2)}{E_2}+\ldots+ {E_n}\right)\right)^{\frac{1}{\gamma-1}}   
\left(2^{\frac{\gamma-2}{\gamma-1}}\right)^{-n}
\,,
\end{multline*}
where the second step uses \(r_n<r_{n-1}/2\) and \(r_1<r_0\). We can use the conditional Jensen's inequality for the concave function
\(u\mapsto u^{\frac{1}{\gamma-1}}\) (concave because \(\gamma>2\)) to deduce that the mean total time of travel, conditional on \(V_0\)
(equivalently \(S_0\)), is bounded above by
\begin{equation}\label{eqn:time-of-travel}
 \frac{5r_0^{\frac{\gamma-2}{\gamma-1}}}{4} 
\sum_{n=1}^\infty \left(
2^{-(n-1)}r_0S_0+
 \frac{8}{\sqrt{5}-2}\right)^{\frac{1}{\gamma-1}}   
\left(2^{\frac{\gamma-2}{\gamma-1}}\right)^{-n}\,.
\end{equation}
Comparison with a geometric sum shows that this sum is finite, since \(\gamma>2\).

We deduce the finiteness of the conditional mean time from \(x_1\) to \(x_2\), since the path can be completed by extending either \(\ell_1\) or its counterpart
in the \(x_2\) path construction; the extra length required is bounded above by \(r_1<r_0\), and the extra time required is therefore bounded above by \(V_0^{-1} r_0\).

Finally, finiteness of mean length follows by multiplying the conditional mean time by \(V_0=S_0^{-\tfrac{1}{\gamma-1}}\) and then taking the expectation. The decisive calculation concerns 
what happens to the conditional bound \eqref{eqn:time-of-travel} when multiplying through by \(S_0^{-\tfrac{1}{\gamma-1}}\)
and taking the expectation; we obtain a mean length upper bound of
\begin{multline}\label{eqn:length-of-travel}
\Expect{  \left(\frac{1}{S_0}\right)^{\frac{1}{\gamma-1}}   
 \frac{5r_0^{\frac{\gamma-2}{\gamma-1}}}{4} 
\sum_{n=1}^\infty \left(
2^{-(n-1)}r_0 S_0+
 \frac{8}{\sqrt{5}-2}\right)^{\frac{1}{\gamma-1}}   
\left(2^{\frac{\gamma-2}{\gamma-1}}\right)^{-n}
 }
\quad=\quad\\
\quad=\quad
 \frac{5r_0^{\frac{\gamma-2}{\gamma-1}}}{4} 
\sum_{n=1}^\infty
\Expect{ \left(\frac{1}{S_0}\right)^{\frac{1}{\gamma-1}}
 \left(
2^{-(n-1)}r_0 S_0+
 \frac{8}{\sqrt{5}-2}\right)^{\frac{1}{\gamma-1}} 
 }
\left(2^{\frac{\gamma-2}{\gamma-1}}\right)^{-n}
\quad\leq\quad\\
\quad\leq\quad
 \frac{5r_0^{\frac{\gamma-2}{\gamma-1}}}{4} 
\sum_{n=1}^\infty
\Expect{ 
 \left(
2^{-(n-1)}r_0 S_0+
 \frac{8}{\sqrt{5}-2}\right)^{\frac{1}{\gamma-1}} \;;\; S_0\geq1
 }
\left(2^{\frac{\gamma-2}{\gamma-1}}\right)^{-n}
+\\
+
 \frac{5r_0^{\frac{\gamma-2}{\gamma-1}}}{4} 
\Expect{ \left(\frac{1}{S_0}\right)^{\frac{1}{\gamma-1}}\;;\;S_0<1}
\sum_{n=1}^\infty
 \left(
2^{-(n-1)}r_0 +
 \frac{8}{\sqrt{5}-2}\right)^{\frac{1}{\gamma-1}} 
\left(2^{\frac{\gamma-2}{\gamma-1}}\right)^{-n}
\,.
\end{multline}
Finiteness of the first summand follows by using the conditional Jensen's inequality as before (noting that \(\gamma>2\)).
Finiteness of the second summand follows by noting, as \(\gamma>2\),
\[
\Expect{ \left(\frac{1}{S_0}\right)^{\frac{1}{\gamma-1}}\;;\;S_0<1}\quad<\quad\infty\,.
\]
\end{proof}

We can now prove the full result: the \(\Pi\)-geodesics between specified points are of finite mean length if \(d=2\).
\begin{thm}\label{thm:finiteness-of-mean}
Suppose \(d=2\) and \(\gamma>2\).
 Consider a \(\Pi\)-geodesic \(\xi\) connecting two points \(x_1\) and \(x_2\). The mean length of \(\xi\) is finite.
\end{thm}

\begin{proof}
 Consider two points \(x_1\), \(x_2\). Without loss of generality, set \(r_0>\tfrac{3}{\sqrt2}r_1=\tfrac{3}{\sqrt2}|x_2-x_1|\), \(\half(x_1+x_2)=\origin\), and \(x_1\), \(x_2\in\ball(\origin, r_0)\). 
Note that we can pick \(r_0\) as large as we please.
We wish to show that the \(\Pi\)-geodesic \(\xi\) from \(x_1\) to \(x_2\) is of finite length.

It is immediate from the \(\Pi\)-geodesic property that the time spent by \(\xi\) in \(\ball(\origin, r_0)\) cannot exceed the time spent travelling from \(x_1\) to \(x_2\) 
using the path described in Lemma \ref{lem:finite-mean-length-in-ball}. 
Following the arguments of Lemma \ref{lem:finite-mean-length-in-ball}, we deduce finiteness of mean for the length of the portion of \(\xi\) 
lying in \(\ball(\origin, r_0)\).

Let \(V_0=S_0^{-\frac{1}{\gamma-1}}\) be the fastest line hitting  \(\ball(\origin, r_0)\). Recall from Equation \eqref{eq:fastest-line} of Lemma \ref{lem:finite-mean-length-in-ball}
that
\[
 S_0 \quad=\quad \frac{1}{\pi r_0} E_0 \qquad \text{ where \(E_0\) is distributed as }\text{Exponential}(1)\,.
\]

\begin{Figure}
  \includegraphics[width=2in]{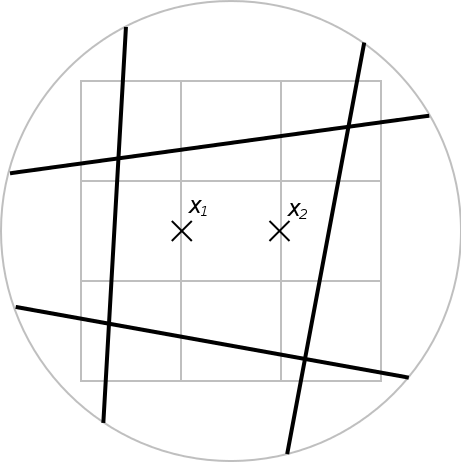}
\centering
\caption{\label{fig:racetrack}
Illustration of the racetrack construction: four \(r_1\times 3r_1\) rectangles placed to surround a central \(r_1\times r_1\) square,
which is centred inside a disc of radius \(r_0> \tfrac{{3}}{\sqrt2}r_1\). The racetrack is formed by four lines connecting the short sides of each rectangle,
chosen to be the fastest such lines which are strictly slower than the fastest line hitting the disc.
}
\end{Figure}
On the other hand, consider the ``racetrack'' around \(\origin\) formed by the fastest lines slower than \(V_0\) and connecting the short sides of rectangles
of sides \(r_1\) and \(3r_1\), placed to surround a central \(r_1\times r_1\) square
(see Figure \ref{fig:racetrack}).
By our choice of \(r_0\), the rectangles are all contained in \(\ball(\origin,r_0)\).
Each of these lines intersects the \(3r_1\times3r_1\) square in a segment of length at most \(\sqrt{10}r_1\). Moreover the invariant line measure of the set of lines
joining the short sides of a rectangle
of sides \(r_1\) and \(3r_1\) is given by
\[
 \half\left(
2\times \sqrt{10}r_1 - 2\times 3r_1
\right)\quad=\quad (\sqrt{10}-3)r_1\,.
\]
Therefore the speed-limits \(V'_i=(S'_i)^{-\frac{1}{\gamma-1}}\) (\(i=1,2,3,4\)) of these lines have distributions given by
\[
  S'_i \quad=\quad S_0 + \frac{1}{(\sqrt{10}-3) r_1} E'_i \qquad \text{ where \(E'_i\) is distributed as }\text{Exponential}(1)\,.
\]
Here the \(E'_1\), \(E'_2\), \(E'_3\), \(E'_4\) are independent of each other and of \(S_0\); 
this can be argued based on the facts that they are based on line-sets which are disjoint and conditioned on being slower than \(S_0\).
The racetrack establishes a path of length at most \(4\sqrt{10}r_1\), which can be traversed in time at most
\begin{multline}\label{eqn:max-time}
 T_* \;=\; \sqrt{10}r_1 \sum_{i=1}^4 \left(S_0 + \frac{1}{(\sqrt{10}-3) r_1} E'_i\right)^{\frac{1}{\gamma-1}}
\;=\;
\sqrt{10}r_1^{\frac{\gamma-2}{\gamma-1}} \sum_{i=1}^4 \left(r_1 S_0 + \frac{1}{\sqrt{10}-3} E'_i\right)^{\frac{1}{\gamma-1}}
\\
\quad\leq\quad
\sqrt{10}r_1^{\frac{\gamma-2}{\gamma-1}} \left(4 r_1^{\frac{1}{\gamma-1}} S_0^{\frac{1}{\gamma-1}} + \sum_{i=1}^4 \left(\frac{1}{\sqrt{10}-3 } E'_i\right)^{\frac{1}{\gamma-1}}\right)
\\
\quad=\quad
4\sqrt{10}r_1 S_0^{\frac{1}{\gamma-1}}
+
\sqrt{10}r_1^{\frac{\gamma-2}{\gamma-1}}\sum_{i=1}^4 \left(\frac{1}{\sqrt{10}-3 } E'_i\right)^{\frac{1}{\gamma-1}}
\,,
\end{multline}
where the inequality follows from the Minkowski inequality (note that \(\gamma>2\)).
It follows that \(\xi\) cannot spend more than \(T_*\) of time outside of \(\ball(\origin, r_0)\), since otherwise 
it would be possible to take a short-cut involving only some of the racetrack and two portions of \(\xi\) lying within \(\ball(\origin, r_0)\), thus travelling from 
\(x_1\) to \(x_2\) in less time overall.

We now apply the comparison technique used in the proof of Theorem \ref{thm:a-priori-bound}, using a scalar comparison process \(y\). We suppose that \(\xi\) starts at 
\(\partial\ball(\origin, r_0)\), so \(\dist(\xi(0),\origin)=r_0\). Then \(|\xi|<y\), where \(y(0)=0\) and \(y'(t)=\overline{V}(y(t))\). Note that the fastest line hitting \(\ball(\origin, r_0)\)
has speed-limit \(V_0\), so \(\overline{V}(y(0))=V_0\). Moreover \(\overline{V}(r)\) for \(r>r_0\) is based entirely on lines with speeds faster than \(V_0\),
and is therefore independent of \(E'_1\), \(E'_2\), \(E'_3\), \(E'_4\).

In dimension \(d=2\), generalized distance is simply ordinary distance. So the recursive formulation \eqref{eqn:meta-slowness-recursion} becomes
\begin{align}\label{eqn:meta-slowness-recursion2}
 R_n-R_{n-1} \quad&=\quad \frac{1}{S_{n-1}}\;\text{Exponential}\left(\pi\right)\,,\\
       S_n   \quad&=\quad S_{n-1} U_n \,,
\end{align}
for independent Uniform\((0,1)\) random variables \(U_i\),
with 
distribution of
\(S_0\) 
as above.

The times between successive changes of speed are given by 
\[
S_{n-1}^{\frac{1}{\gamma-1}}(R_n-R_{n-1})
\quad=\quad
{S_{n-1}^{-(\gamma-2)/(\gamma-1)}}\text{Exponential}\left(\pi\right)\,. 
\]
We know that \(S_n\) decreases as \(n\to\infty\).
Accordingly, a coupling argument shows that the number \(N_{T_*}\) of changes of speed by time \(T_*\) will not exceed \(\widetilde{N}\), where
\(\widetilde{N}\) has distribution \(\text{Poisson}\left({\pi S_0^{\frac{\gamma-2}{\gamma-1}} T_*}\right)\) when conditioned on \(S_0\) and \(T_*\),
and is independent of the actual changes of speed (though not of \(T^*\) or \(S_0\)).
Thus the final speed is no more than
\[
 S_0^{-\frac{1}{\gamma-1}} \prod_{n=1}^{\widetilde{N}} U_n^{-\frac{1}{\gamma-1}}\,,
\]
and the distance travelled by the \(\Pi\)-geodesic outside \(\ball(\origin, r_0)\) cannot exceed
\[
 S_0^{-\frac{1}{\gamma-1}} T_* \prod_{n=1}^{\widetilde{N}} U_n^{-\frac{1}{\gamma-1}}\,.
\]
Conditioning on \(E'_1\), \(E'_2\), \(E'_3\), \(E'_4\), and \(S_0\), we can
integrate out first the \(U_i\)'s and then the Poissonian variation \(\widetilde{N}\) from the resulting bound on mean distance travelled, and then use
 the upper bound on \(T_*\) specified by \eqref{eqn:max-time}.
We thus obtain the following bound on mean distance travelled, using \(S_0=\tfrac{1}{\pi r_0} E_0\) and exchangeability of the \(E'_i\):
\begin{multline*}
 \Expect{S_0^{-\frac{1}{\gamma-1}} T_* \times \prod_{n=1}^{\widetilde{N}} U_n^{-\frac{1}{\gamma-1}}}=
 \Expect{S_0^{-\frac{1}{\gamma-1}} T_* \times \left(\frac{\gamma-1}{\gamma-2}\right)^{\widetilde{N}} }
=
\Expect{S_0^{-\frac{1}{\gamma-1}} T_* \times \exp\left( \frac{\pi S_0^{\frac{\gamma-2}{\gamma-1}} }{\gamma-2}T_*\right)}\\
\quad\leq\quad
4\sqrt{10}\;\mathbb{E}\Big[
\left(r_1 
+
S_0^{-\frac{1}{\gamma-1}}r_1^{\frac{\gamma-2}{\gamma-1}}\left(\frac{1}{\sqrt{10}-3 } E'_1\right)^{\frac{1}{\gamma-1}}
\right)
\times
\exp\left( \frac{4\sqrt{10} \pi  r_1 S_0}{\gamma-2}\right)\times\\
\times
\exp\left( \frac{\sqrt{10} \pi (r_1 S_0)^{\frac{\gamma-2}{\gamma-1}} }{\gamma-2}\sum_{i=1}^4 \left(\frac{1}{\sqrt{10}-3 } E'_i\right)^{\frac{1}{\gamma-1}}\right)
\Big]\\
\quad\leq\quad
4\sqrt{10}\; r_1\;\mathbb{E}\Big[
\left(1 
+
\left(\frac{\pi}{\sqrt{10}-3 } \right)^{\frac{1}{\gamma-1}} \left(\frac{r_0}{r_1}\right)^{\frac{1}{\gamma-1}} \left(\frac{E'_1}{E_0}\right)^{\frac{1}{\gamma-1}}
\right)
\times
\exp\left( \frac{4\sqrt{10}}{\gamma-2}\; \frac{r_1}{r_0} E_0\right)\times\\
\times
\exp\left(\frac{\sqrt{10}}{\gamma-2} \frac{\pi^{\frac{1}{\gamma-1}}  }{(\sqrt{10}-3)^{\frac{1}{\gamma-1}}}
\left(\frac{r_1}{r_0}\right)^{\frac{\gamma-2}{\gamma-1}}
\sum_{i=1}^4 E_0^{\frac{\gamma-2}{\gamma-1}} (E'_i)^{\frac{1}{\gamma-1}} 
  \right)
\Big]\,.
\end{multline*}
But now we can apply the simple inequality
\[
 (E_0)^{\frac{\gamma-2}{\gamma-1}}(E'_i)^{\frac{1}{\gamma-1}}\quad=\quad
 (E_0)^{1-\frac{1}{\gamma-1}}(E'_i)^{\frac{1}{\gamma-1}} \quad\leq\quad
E_0 + E'_i\qquad\text{(for \(E_0>0\), \(E'_i>0\))}\,,
\]
to deduce that
\begin{multline}\label{eqn:controlling-inequality}
  \Expect{S_0^{-\frac{1}{\gamma-1}} T_* \times \prod_{n=1}^{N_{T_*}} U_n^{-\frac{1}{\gamma-1}}}\quad\leq\quad
4\sqrt{10}\; r_1\;\mathbb{E}\Big[
\left(1 
+
\left(\frac{\pi}{\sqrt{10}-3 } \right)^{\frac{1}{\gamma-1}} \left(\frac{r_0}{r_1}\right)^{\frac{1}{\gamma-1}} \left(\frac{E'_1}{E_0}\right)^{\frac{1}{\gamma-1}}
\right)
\times
\\
\times
\exp\left( \frac{4\sqrt{10}}{\gamma-2}\; \frac{r_1}{r_0} E_0\right)\times
\exp\left(\frac{\sqrt{10}}{\gamma-2} \frac{\pi^{\frac{1}{\gamma-1}}  }{(\sqrt{10}-3)^{\frac{1}{\gamma-1}}}
\left(\frac{r_1}{r_0}\right)^{\frac{\gamma-2}{\gamma-1}}
\left(4 E_0 + E'_1 + E'_2 + E'_3 + E'_4\right)
  \right)
\Big]\,.
\end{multline}
Now the expectation can be bounded above by an expression involving finite Gamma integrals of the forms
\[
 \int_0^\infty \exp(-\beta u)\d{u}\,,\qquad  \int_0^\infty u^{\frac{1}{\gamma-1}} \exp(-\beta u)\d{u}\,,\qquad  \int_0^\infty u^{-\frac{1}{\gamma-1}} \exp(-\beta u)\d{u}\,,
\]
for \(\beta>0\) (once \(r_0\) is chosen sufficiently large) and \(\gamma>2\).
Consequently the mean length of the \(\Pi\)-geodesic outside of \(\ball(\origin,r_0)\) must also be finite, proving
the theorem.
\end{proof}

This work shows that that planar
spatial networks formed from the Poisson line process model satisfy
property \ref{def:SIRSN-item-finite-length} of Definition \ref{def:SIRSN}.

 \section[Further properties of \texorpdfstring{$\Pi$}{Π}-geodesics in d=2]{Further properties of \(\Pi\)-geodesics in \(d=2\)}\label{sec:properties}
Finally we show that in dimension \(d=2\) any specified point \(x\) almost surely possesses just one \(\Pi\)-geodesic to \(\infty\);
moreover that for any three distinct points \(x,y,z\in\Reals^2\) almost surely the \(\Pi\)-geodesics from \(x\) to \(y\) and from \(x\) to \(z\) coincide for a non-trivial initial segment;
and also that if \(\Xi\) is an independent Poisson point process in \(\Reals^2\) 
then almost surely the totality of all \(\Pi\)-geodesics between points of \(\Xi\) forms a fibre process \citep[\S8.3]{ChiuStoyanKendallMecke-2013}
which places finite total length in any given compact subset of \(\Reals^2\).
 
This last result shows that the network generated by \(\Pi\) possesses a very weak variant of Aldous' \emph{SIRSN property} \citep{AldousGanesan-2013,Aldous-2012}; a SIRSN
(scale-invariant random spatial network) would have the property that the \emph{mean} total length per unit area was finite (weak SIRSN property)
and moreover the property that the mean total length of connecting routes of distance at least \(1\) from start and source would remain bounded as
the intensity of \(\Xi\) increased to infinity. It is conjectured that the network generated by \(\Pi\) is a true SIRSN, 
but at present all we can prove is the above ``pre-SIRSN'' property.
 
All three of these results depend on the same construction, based on \citet[Figure 6]{Aldous-2012}:
consider the behaviour of \(\Pi\)-geodesics starting from points in a \(2\times2\) square centred on the origin and ending outside a \(10\times10\) square centered at \((0,2)\).
Condition on the sides of the two squares and the \(y\)-axis all being subsets of lines from \(\Pi\), with speeds as follows: 
the sides of the \(2\times2\) square have speed \(a\), 
the \(y\)-axis has speed \(b\),
and the sides of the \(10\times10\) square have speed \(c\).
Suppose further that no other lines of \(\Pi_1\) (speed exceeding \(1\)) hit the \(10\times10\) square.
Figure \ref{fig:construction} illustrates the construction.
\begin{Figure}
  \includegraphics[width=2in]{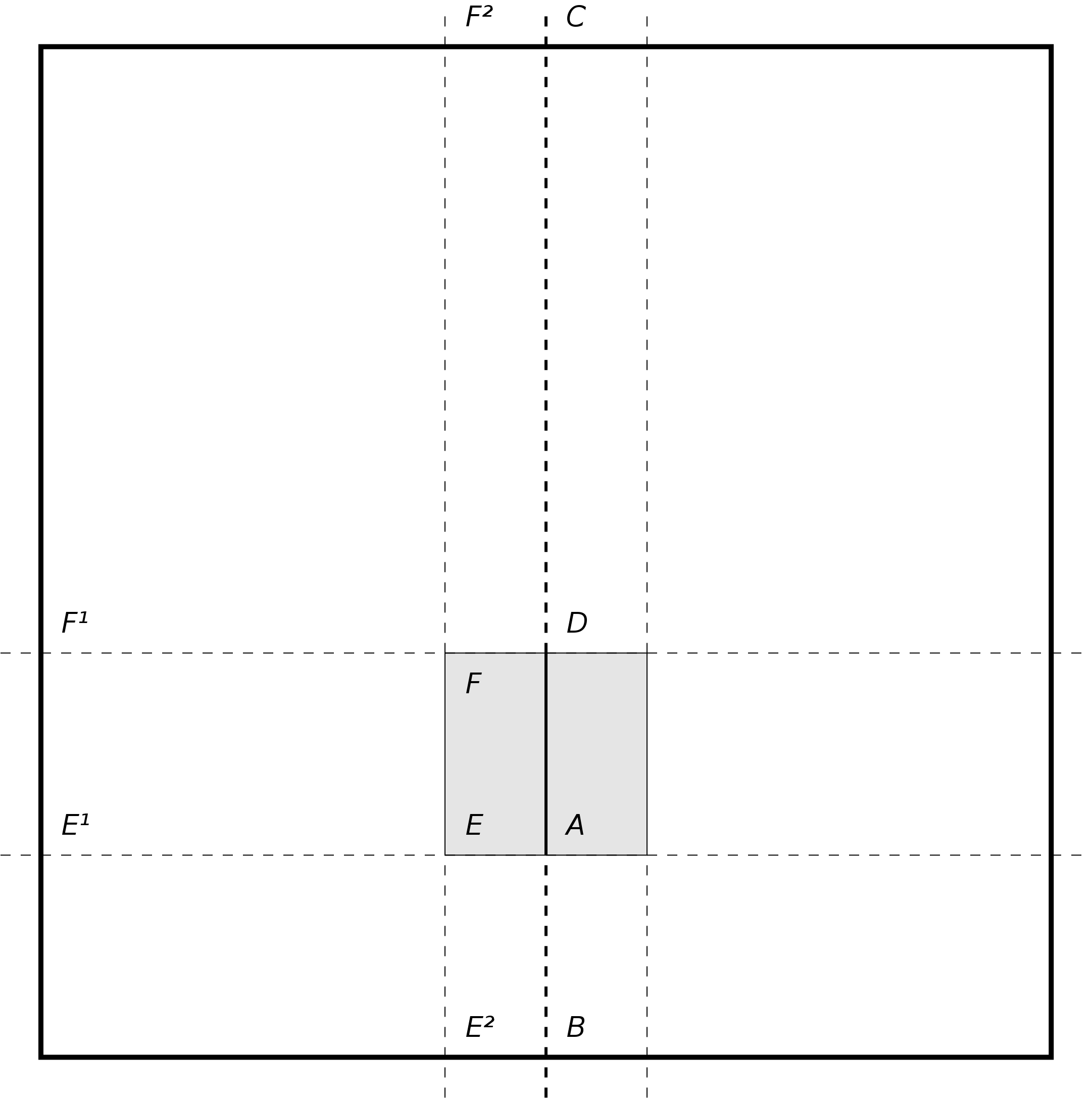}
\centering
\caption{\label{fig:construction}
Construction forcing certain \(\Pi\)-geodesics to pass through the points \(A\) and \(B\).
}
\end{Figure}

\begin{lemma}\label{lem:pre-SIRSN-structure}
 In the above situation, suppose that \(c > 10b > 59 a/3 > 354 / 3\). 
 Then any \(\Pi\)-geodesic connecting the interior of the \(2\times2\) square and the exterior of the \(10\times10\) square must pass through the points \(A=(0, -1)\)
 and \(B=(0,-3)\).
\end{lemma}

\begin{proof}
To simplify exposition, we can and shall confine our attention to \(\Pi\)-geodesics constrained to lie in or on the \(10\times10\) square.

 First note from the figure that the construction can be divided into rectangles of dimensions \(2\times2\), \(1\times2\), \(1\times6\), \(2\times4\), and \(6\times4\).
For each of these rectangles the sides have speed at least \(c\), while any other lines intersecting the rectangles have speed not exceeding \(1\). 
Geometric comparisons show that, for any of these squares, \(\Pi\)-geodesics between pairs of points on the perimeter cannot intersect the interior. 
This is a ``no short-cut'' condition for \(\Pi\)-geodesics. 
In particular, a (constrained) \(\Pi\)-geodesic from the \(2\times2\) square to point \(C\) must be confined to the union of the \(2\times2\) square and the other square boundaries.

Consequently, such a \(\Pi\)-geodesic must have a final segment which is one of
\begin{enumerate}
 \item \(A\to B\to C\) (the \(B\to C\) part using the perimeter of the \(10\times10\) square);
 \item \(D\to C\);
 \item \(E\to E^2\to C\) (last part using perimeter);
 \item \(E\to E^1\to C\) (last part using perimeter);
 \item \(F\to F^2\to C\) (last part using perimeter);
 \item \(F\to F^1\to C\) (last part using perimeter);
 \item or one of four cases which are mirror images of cases \(3-6\).
\end{enumerate}

We can now compare times taken by these alternative routes:
under the condition \(c > 10b > 59 a/3 > 354 / 3\) it transpires that the quickest route always passes through the points \(A\)  and \(B\) as required.
\end{proof}

This lemma enables soft proofs of the three theorems of this section.
\begin{thm}\label{thm:unique-to-infinity}
 Suppose \(\gamma>d=2\). With probability \(1\), for any point \(x\) there is one and only one \(\Pi\)-geodesic from \(x\) to \(\infty\); moreover all such infinite \(\Pi\)-geodesics eventually coalesce when sufficiently far away from the origin.
\end{thm}
\begin{proof}
Small perturbations of the structure described in Lemma \ref{lem:pre-SIRSN-structure} will have the same property (\(\Pi\)-geodesics from within small squares to exteriors of large squares all pass through specified points), 
and so there is a positive probability \(\eps>0\) that \(\Pi\) will generate a structure ensuring that all \(\Pi\) geodesics from the \(2\times2\) square reaching out further than the \(10\times10\) square
will have to pass through a specified pair of points near to \(A\) and \(B\).

Moreover we can use scale-invariance to generate further structures at larger scales, such that whether or not a corresponding perturbation of each structure is realized is independent of whether or not the other structures are realized. 
An appeal to the second Borel-Cantelli lemma then shows that there must be an infinite sequence of planar points \((0, -3a_1)\), \((0, -3a_2)\), \ldots \(\to\infty\), such that if 
\(x\in[-a_n,a_n]^2\) and \(y\not\in[-5a_n,5a_n]\times[-3a_n,7a_n]\)
then any \(\Pi\)-geodesic from \(y\) to \(x\) must pass through \((0, -3a_n)\). Moreover the section of this \(\Pi\)-geodesic from \(x\) to \((0,-3a_n)\) is uniquely determined, since for almost all \(y\) the \(\Pi\)-geodesic from \(y\) to \(x\) will be unique (Theorem \ref{thm:uniqueness}). Successive sections of similar \(\Pi\)-geodesics therefore build up a unique \(\Pi\)-geodesic from \(x\) to \(\infty\). Moreover the nature of the structure described in Lemma \ref{lem:pre-SIRSN-structure} ensures that, for any other planar point \(y\), this other point will be included in the smaller of the two rectangles of structures at sufficiently large scales: eventual coalescence of all infinite \(\Pi\)-geodesics thus follows.
\end{proof}

We can now establish a result similar to that of \cite{Bettinelli-2014} for the planar Brownian map: almost surely \(\Pi\)-geodesics emanating from a given point must initially coalesce.
(Evidently this cannot hold for \emph{all} points: consider points actually lying on a \(\Pi\)-geodesic!)
\begin{thm}\label{thm:coalescence-of-geodesics}
 Suppose \(\gamma>d=2\). Almost surely for any distinct points \(x\), \(y\), \(z\), the \(\Pi\)-geodesics from \(x\) to \(y\) and from \(x\) to \(z\) coincide for a non-trivial initial segment.
\end{thm}
\begin{proof}
 The argument follows that of Theorem \ref{thm:unique-to-infinity}, except that structures are now generated at increasingly smaller scales, all surrounding \(x\).
\end{proof}

\begin{thm}\label{thm:pre-SIRSN}
 Suppose \(\gamma>d=2\). The network generated by \(\Pi\) has the \emph{pre-SIRSN property}, in the sense that if \(\Xi\) is an independent Poisson point process in \(\Reals^2\) 
then almost surely the totality of all \(\Pi\)-geodesics between points of \(\Xi\) intersected with a compact set has finite total length.
\end{thm}

\begin{proof}
 By scaling and monotonicity, we may suppose that the compact set in question is the \(1\times1\) square centred at the origin. Arguing as in Theorem \ref{thm:coalescence-of-geodesics},
at a suitably large scale there will be a structure which forces all \(\Pi\)-geodesics between points in the \((1\times1)\) square and points in the exterior of a \(L\times L\) square to pass through a specified point \(H\)
near the boundary of the \(L\times L\) square. Here \(L\) is random but depends only on \(\Pi\), not \(\Xi\).

Let \(N\) be the random number of points placed by \(\Xi\) in the \(L\times L\) square. 
Then at most \(\binom{N+1}{2}\) \(\Pi\)-geodesics can intersect the \(1\times1\) square (based on these points and on \(H\)).
Each of these \(\Pi\)-geodesics has finite length (Theorem \ref{thm:finiteness-of-mean}), so the result follows.
\end{proof}
This proves property \ref{def:SIRSN-item-locally-finite} of Definition \ref{def:SIRSN} for the planar case.

 \section{Conclusion}\label{sec:conclusion}
This paper has established:
\begin{enumerate}
 \item Basic metric space properties of the Poisson line process model,
 including existence of minimum-time paths;
 \item Extension of the metric space properties of the Poisson line process model to higher dimensions;
 \item Approximation results for minimum-time paths (``\(\Pi\)-geodesics'');
 \item Almost-sure uniqueness and finite mean length of \(\Pi\)-geodesics in the planar case;
 \item Local finiteness of resulting networks in the planar case.
\end{enumerate}
As a result, it follows that the planar Poisson line process model produces a pre-SIRSN (Definition \ref{def:SIRSN}).
The major outstanding question is, whether in fact the weak SIRSN or even full SIRSN properties hold for the planar Poisson line process model.
Extending the method of Theorem \ref{thm:pre-SIRSN} would require a much more quantitative approach; 
it would be necessary to estimate the scale at which
there would exist structures forcing large \(\Pi\)-geodesics to pass through specified points.

A linked question concerns the nature of \(\Pi\)-geodesics in the planar case: can they be represented using sequences of line segments from the Poisson line process,
or do they necessarily involve the tree-like representations described in the proofs of Theorems \ref{thm:connection} and \ref{thm:metric-space}?
The methods of Section \ref{sec:pi-geodesics-uniqueness} are suggestive that the answer is yes, 
but do not entirely exclude the possibility of slow-down at points not lying on \(\Pi\).  One must show that \(\Pi\)-geodesics
between pairs of points can \emph{never} slow down to zero speed \emph{en route}.
This is conceptually linked to the notion of network transit points \citep{BastFunkeSandersSchultes-2007}, as discussed in \cite{Aldous-2012}.

A further question concerns whether the pre-SIRSN property extends to higher dimensions. 
Point-line duality is pervasive in the arguments of the second half of this paper, so presently it is not clear how to proceed with this.
Possibly a quantitative form of \(\Pi\)-geodesic coalescence, as described in Theorem \ref{thm:coalescence-of-geodesics}, might allow headway to be made.

The paper has focussed throughout on results relating directly to the pre-SIRSN property of the Poisson line process model.
We have noted in passing and without proof some computations which establish sharpness of conditions on \(\gamma\) in our results
(Remarks \ref{rem:equivalence}, \ref{rem:connection-by-geodesics}); 
similar calculations show that the topology of Euclidean space \(\Reals^d\) viewed as a \(\Pi\)-geodesic metric space (for \(\gamma>d\)) is the normal Euclidean topology.

Finally, we note that an alternative motivation for the above work is given by recent developments in the study of Brownian maps; 
the random metric space given here can be compared to the Brownian map (note for example that both situations exhibit coalescence of geodesics)
and promises by its constructive nature to be (relatively) more amenable to rigorous mathematical investigation, as well as providing higher dimensional constructions.
It would be of great interest to clarify the extent to which the two theories can be linked.
Note in particular the intriguing prospect of mimicking the Brownian map theory by the construction of ``Liouville Brownian motions'', perhaps using Dirichlet form theory
(compare \citealp{Berestycki-2013,GarbaRhodesVargas-2013}).

\bigskip
\noindent
\textbf{Acknowledgements:} My thanks to David Aldous, who challenged me to try to understand the Poisson line process model for a scale-invariant network.

\bibliographystyle{chicago}
\bibliography{SIRSN-account}



\end{document}